\documentclass{amsart}
\usepackage[latin9]{inputenc}
\usepackage{amstext}
\usepackage{amsthm}

\makeatletter
\numberwithin{equation}{section}
\numberwithin{figure}{section}

\usepackage{amsfonts}

\usepackage{amscd}
\usepackage{thmdefs}

\input{tcilatex}
\theoremstyle{definition}
\theoremstyle{remark}
\numberwithin{equation}{section}

\input tcilatex

\makeatother

\begin{document}

\title[Supersingular curves on Picard modular surfaces]{Supersingular curves on Picard modular surfaces modulo an inert
prime}

\author{Ehud de Shalit}

\address{Hebrew University, Jerusalem, Israel}

\email{deshalit@math.huji.ac.il.}

\author{Eyal Z. Goren}

\address{McGill University, Montreal, Canada}

\email{eyal.goren@mcgill.ca}

\date{February 22, 2015}
\begin{abstract}
We study the supersingular curves on Picard modular surfaces modulo
a prime $p$ which is inert in the underlying quadratic imaginary
field. We analyze the automorphic vector bundles in characteristic
$p,$ and as an application derive a formula relating the number of
irreducible components in the supersingular locus to the second Chern
class of the surface. 
\end{abstract}

\maketitle
\bigskip{}
\,

The characteristic $p$ fibers of Shimura varieties of PEL type admit
special subvarieties over which the abelian varieties which they parametrize
become supersingular. These special subvarieties and the automorphic
vector bundles over them have become the subject of extensive research
in recent years. The purpose of this note is to analyze one of the
simplest examples beyond Shimura curves, namely Picard modular surfaces,
at a prime $p$ which is inert in the underlying quadratic imaginary
field. The supersingular locus on these surfaces has been studied
in depth by Bültel, Vollaard and Wedhorn ({[}6{]},{[}31{]},{[}32{]})
and our debt to these authors should be obvious. It consists of a
collection of Fermat curves of degree $p+1,$ intersecting transversally
at the \emph{superspecial} points. We complement their work by analyzing
the automorphic vector bundles in characteristic $p,$ in relation
to the supersingular strata. Although some of our results (e.g. the
construction of a Hasse invariant, see {[}12{]}) have been recently
generalized to much larger classes of Shimura varieties, focusing
on this simple example allows us to give more details. As an example,
we derive the following theorem (\ref{ss curves}).

\begin{theorem} Let $\mathcal{K}$ be a quadratic imaginary field
and let $\bar{S}$ be the Picard modular surface of level $N\ge3$
associated with $\mathcal{K}$ (for a precise definition see the text
below). Let $p$ be a prime which is inert in $\mathcal{K}$ and relatively
prime to $2N.$ Then the number of irreducible components in the supersingular
locus of $\bar{S}\func{mod}p$ is $c_{2}(\bar{S})/3,$ where $c_{2}(\bar{S})$
is the second Chern class of $\bar{S}.$ \end{theorem}

Thanks to results of Holzapfel, this number is expressed in terms
of the value of the $L$-function $L(s,(D_{\mathcal{K}}/\cdot))$
at $s=3,$ see Theorem \ref{Holz}.

Although the uniformization of the supersingular locus by Rapoport-Zink
spaces yields a group-theoretic parametrization of the irreducible
components by a certain double coset space, this parametrization in
itself is not sufficient to yield the theorem. Our proof goes along
different lines, invoking intersection theory on $\bar{S}.$

Note that the number of irreducible components comes out to be independent
of $p.$ A similar result was obtained in {[}2{]} for Hilbert modular
surfaces. This is different from the situation with the Siegel modular
variety $A_{g}.$ We are still lacking a conceptual understanding
of this independence of $p$.

\medskip{}

The first section introduces notation and background. The second contains
most of the analysis modulo $p.$ The Picard surface modulo $p$ has
3 strata: the $\mu$-ordinary (generic) stratum $\bar{S}_{\mu}$,
the general supersingular locus $S_{gss}$ and the superspecial points
$S_{ssp}$. There are two basic automorphic vector bundles to consider,
a rank 2 bundle $\mathcal{P}$ and a line bundle $\mathcal{L}.$ Of
particular importance are the Verschiebung homomorphisms $V_{\mathcal{P}}:\mathcal{P}\rightarrow\mathcal{L}^{(p)}$
and $V_{\mathcal{L}}:\mathcal{L}\rightarrow\mathcal{P}^{(p)}.$ It
turns out that outside the superspecial points $V_{\mathcal{P}}$
and $V_{\mathcal{L}}$ are both of rank 1, but that the supersingular
locus is \emph{defined} by the equation $V_{\mathcal{P}}^{(p)}\circ V_{\mathcal{L}}=0.$
The Hasse invariant $h_{\bar{\Sigma}}=$ $V_{\mathcal{P}}^{(p)}\circ V_{\mathcal{L}}$
is a canonical global section of $\mathcal{L}^{p^{2}-1}$ whose divisor
is $S_{ss}=S_{gss}\cup S_{ssp}.$ On $S_{ss}$, in turn, we construct
a canonical section $h_{ssp}$ of $\mathcal{L}^{p^{3}+1},$ which
vanishes precisely at the superspecial points $S_{ssp}$ (to a high
order). This secondary Hasse invariant is related to more general
work of Boxer {[}5{]}.

In the last section we use the information carried by $h_{ssp},$
together with some intersection theoretic computations, to prove the
above theorem.

We would like to thank the referee for a thorough reading of the manuscript.

\section{Background}

\subsection{Notation}

We recall some classical facts about unitary groups in three variables
and set up some notation.

\subsubsection{The quadratic imaginary field}

Let $\mathcal{K}$ be an imaginary quadratic field, contained in $\Bbb{C}.$
We denote by $\Sigma:\mathcal{K}\hookrightarrow\Bbb{C}$ the inclusion
and by $\bar{\Sigma}:\mathcal{K}\hookrightarrow\Bbb{C}$ its complex
conjugate. We use the following notation:
\begin{itemize}
\item $d_{\mathcal{K}}$ - the square free integer such that $\mathcal{K}=\Bbb{Q}(\sqrt{d_{\mathcal{K}}}).$
\item $D_{\mathcal{K}}$ - the discriminant of $\mathcal{K}$, equal to
$d_{\mathcal{K}}$ if $d_{\mathcal{K}}\equiv1\func{mod}4$ and $4d_{\mathcal{K}}$
if $d_{\mathcal{K}}\equiv2,3\func{mod}4.$
\item $\delta_{\mathcal{K}}=\sqrt{\mathcal{D}_{\mathcal{K}}}$ - the square
root with positive imaginary part, a generator of the different of
$\mathcal{K},$ sometimes simply denoted $\delta.$
\item $\omega_{\mathcal{K}}=(1+\sqrt{d_{\mathcal{K}}})/2$ if $d_{\mathcal{K}}\equiv1\func{mod}4,$
otherwise $\omega_{\mathcal{K}}=\sqrt{d_{\mathcal{K}}},$ so that
$\mathcal{O}_{\mathcal{K}}=\Bbb{Z}+\Bbb{Z}\omega_{\mathcal{K}}.$
\item $\bar{a}$ - the complex conjugate of $a\in\mathcal{K}.$
\item $\func{Im}_{\delta}(a)=(a-\bar{a})/\delta$, for $a\in\mathcal{K}.$ 
\end{itemize}
We fix an integer $N\ge3$ (the ``tame level'') and let $R_{0}=\mathcal{O}_{\mathcal{K}}[1/2d_{\mathcal{K}}N].$
This is our \emph{base ring}. If $R$ is any $R_{0}$-algebra and
$M$ is any $R$-module with $\mathcal{O}_{\mathcal{K}}$-action,
then $M$ becomes an $\mathcal{O}_{\mathcal{K}}\otimes R$-module
and we have a canonical \emph{type decomposition} 
\begin{equation}
M=M(\Sigma)\oplus M(\bar{\Sigma})\label{type}
\end{equation}
where $M(\Sigma)=e_{\Sigma}M$ and $M(\bar{\Sigma})=e_{\bar{\Sigma}}M,$
and where the idempotents $e_{\Sigma}$ and $e_{\bar{\Sigma}}$ are
defined by 
\begin{equation}
e_{\Sigma}=\frac{1\otimes1}{2}+\frac{\delta\otimes\delta^{-1}}{2},\,\,\,\,e_{\bar{\Sigma}}=\frac{1\otimes1}{2}-\frac{\delta\otimes\delta^{-1}}{2}.
\end{equation}
Then $M(\Sigma)$ (resp. $M(\bar{\Sigma})$) is the part of $M$ on
which $\mathcal{O}_{\mathcal{K}}$ acts via $\Sigma$ (resp. $\bar{\Sigma}$).
The same notation will be used for sheaves of modules on $R$-schemes,
endowed with an $\mathcal{O}_{\mathcal{K}}$ action. If $M$ is locally
free, we say that it has \emph{type }$(p,q)$ if $M(\Sigma)$ is of
rank $p$ and $M(\bar{\Sigma})$ is of rank $q.$

\subsubsection{The unitary group}

Let $V=\mathcal{K}^{3}$ and endow it with the hermitian pairing 
\begin{equation}
(u,v)=\,^{t}\bar{u}\left(\begin{array}{lll}
 &  & \delta^{-1}\\
 & 1\\
-\delta^{-1}
\end{array}\right)v.
\end{equation}
We identify $V_{\Bbb{R}}$ with $\Bbb{C}^{3}$ ($\mathcal{K}$ acting
via the natural inclusion $\Sigma$). It then becomes a hermitian
space of signature $(2,1).$ Conversely, any $3$-dimensional hermitian
space over $\mathcal{K}$ whose signature at the infinite place is
$(2,1)$ is isomorphic to $V$ after rescaling the hermitian form
by a positive rational number.

Let 
\begin{equation}
\mathbf{G=U}(V,(,))
\end{equation}
be the unitary group of $V,$ regarded as an algebraic group over
$\Bbb{Q}.$ For any $\Bbb{Q}$-algebra $A$ we have 
\begin{equation}
\mathbf{G}(A)=\left\{ g\in GL_{3}(A\otimes\mathcal{K})|\,\,(gu,gv)=(u,v)\,\,\forall u,v\in V_{A}\right\} .
\end{equation}

We write $G=\mathbf{G}(\Bbb{Q}),\,G_{\infty}=\mathbf{G}(\Bbb{R})$
and $G_{p}=\mathbf{G}(\Bbb{Q}_{p}).$

We also introduce an alternating $\Bbb{Q}$-linear pairing $\left\langle ,\right\rangle :V\times V\rightarrow\Bbb{Q}$
defined by $\left\langle u,v\right\rangle =\func{Im}_{\delta}(u,v).$
We then have the formulae 
\begin{equation}
\left\langle au,v\right\rangle =\left\langle u,\bar{a}v\right\rangle ,\,\,\,\,\,\,2(u,v)=\left\langle u,\delta v\right\rangle +\delta\left\langle u,v\right\rangle .
\end{equation}
We call $\left\langle ,\right\rangle $ the \emph{polarization form}.

\subsubsection{The hermitian symmetric domain}

The group $G_{\infty}$ acts on $\Bbb{P}_{\Bbb{C}}^{2}=\Bbb{P}(V_{\Bbb{R}})$
by projective linear transformations, and preserves the open subdomain
$\frak{X}$ of negative definite lines (in the metric $(,)$). This
$\frak{X}$ is biholomorphic to the open unit ball in $\Bbb{C}^{2},$
and $G_{\infty}$ acts on it transitively. Every negative definite
line is represented by a unique vector $^{t}(z,u,1),$ and such a
vector represents a negative definite line if and only if 
\begin{equation}
\lambda(z,u)\overset{\text{def}}{=}\vspace{0.01in}\mathrm{Im}_{\delta}(z)-u\bar{u}>0.
\end{equation}
One refers to the realization of $\frak{X}$ as the set of points
$(z,u)\in\Bbb{C}^{2}$ satisfying this inequality as a \emph{Siegel
domain of the second kind}. It is convenient to think of the point
$x_{0}=$ $(\delta/2,0)$ as the ``center'' of $\frak{X.}$

If we let $K_{\infty}$ be the stabilizer of $x_{0}$ in $G_{\infty}$,
then $K_{\infty}$ is compact and isomorphic to $U(2)\times U(1)$.
Since $G_{\infty}$ acts transitively on $\frak{X},$ we may identify
$\frak{X}$ with $G_{\infty}/K_{\infty}.$\\

\subsection{Picard modular surfaces over $\Bbb{C}$}

We next recall some classical facts about Picard modular surfaces.

\subsubsection{Lattices and their arithmetic groups}

Let $L\subset V$ be the lattice 
\begin{equation}
L=Span_{\mathcal{O}_{\mathcal{K}}}\{\delta e_{1},e_{2},e_{3}\}.
\end{equation}
This $L$ is \emph{self-dual} in the sense that 
\begin{equation}
L=\left\{ u\in V|\,\left\langle u,v\right\rangle \in\Bbb{Z\,\,\forall}v\in L\right\} .
\end{equation}
Equivalently, $L$ is its own $\mathcal{O}_{\mathcal{K}}$-dual with
respect to the hermitian pairing $(,).$

With $N\ge3$ as before let 
\begin{equation}
\Gamma=\Gamma(N)=\left\{ \gamma\in G|\,\gamma L=L\,\text{and }\gamma(u)\equiv u\func{mod}NL\,\,\forall u\in L\right\} .
\end{equation}
Then $\Gamma$ is a discrete torsion-free subgroup of $G$ which acts
freely and faithfully on $\frak{X}.$ As $N\ge3,$ it can be seen
that $\Gamma\subset SU(V,(,)),$ i.e. $\det\gamma=1$ for all $\gamma\in\Gamma.$

\subsubsection{The cusps}

Let $\mathcal{C}$ be the set of isotropic lines in $V.$ Equivalently,
$\mathcal{C}$ may be described as the set of vectors $^{t}(z,u,1)$
with $z,u\in\mathcal{K}$ and $\lambda(z,u)=0,$ together with the
unique cusp at infinity $c_{\infty}=\,^{t}(1,0,0),$ or simply as
the $\mathcal{K}$-rational points on $\partial\frak{X}.$ The group
$\Gamma$ acts on $\mathcal{C}$ and its orbits $\mathcal{C}_{\Gamma}=\Gamma\backslash\mathcal{C}$
are called the \emph{cusps} of $\Gamma.$ The set $\mathcal{C}_{\Gamma}$
is finite.

\subsubsection{Picard modular surfaces and their smooth compactifications}

We denote by $X_{\Gamma}$ the complex surface $\Gamma\backslash\frak{X.}$
Since the action of $\Gamma$ on $\frak{X}$ is free, $X_{\Gamma}$
is smooth. For the following theorem see, for example, {[}7{]}.

\begin{theorem} There exists a smooth projective complex surface
$\bar{X}_{\Gamma}$ containing $X_{\Gamma}$ as a dense open subset.
The irreducible components of $\bar{X}_{\Gamma}-X_{\Gamma}$ are elliptic
curves with complex multiplication by $\mathcal{O}_{\mathcal{K}},$
and are in a one-to-one correspondence with $\mathcal{C}_{\Gamma}.$
These conditions determine $\bar{X}_{\Gamma}$ uniquely. \end{theorem}

Holzapfel studied the Chern classes of the surface $\bar{X}_{\Gamma}.$
He obtained the following result {[}19, (5A.4.3), p.325{]}.

\begin{theorem} \label{Holz}Let $c_{2}(\bar{X}_{\Gamma})$ be the
second Chern class of $\bar{X}_{\Gamma},$ which is equal also to
the Euler characteristic 
\begin{equation}
e(\bar{X}_{\Gamma})=\sum_{i=0}^{4}(-1)^{i}\dim H^{i}(\bar{X}_{\Gamma},\Bbb{C}).
\end{equation}
Then 
\begin{equation}
c_{2}(\bar{X}_{\Gamma})=[\Gamma(1):\Gamma(N)]\cdot\frac{3|D_{\mathcal{K}}|^{5/2}}{32\pi^{3}}L\left(3,\left(\frac{D_{\mathcal{K}}}{\cdot}\right)\right).
\end{equation}
Here $\Gamma(N)=\Gamma$ and 
\begin{equation}
\Gamma(1)=\left\{ \gamma\in G|\,\gamma L=L,\,\det\gamma=1\right\} .
\end{equation}
\end{theorem}

From the functional equation of the $L$-function $L(s,(D_{\mathcal{K}}/\cdot))$
we also get 
\begin{equation}
c_{2}(\bar{X}_{\Gamma})=-[\Gamma(1):\Gamma(N)]\cdot\frac{3}{16}L\left(-2,\left(\frac{D_{\mathcal{K}}}{\cdot}\right)\right).
\end{equation}
Note that the index $[\Gamma(1):\Gamma(N)]=[vol(\Gamma(N)\backslash\frak{X}):vol(\Gamma(1)\backslash\frak{X})],$
\emph{except }if $D_{\mathcal{K}}=-3,$ when $\Gamma(1)$ does not
act faithfully on $\frak{X}$ as it contains the roots of unity of
order 3 in its center. In this case $[\Gamma(1):\Gamma(N)]$ is equal
to $3$ times the volume ratio. This accounts for the factor $\varepsilon$
in {[}19, (5A.4.3){]}.

\subsubsection{The Shimura variety}

The surface $X_{\Gamma}$ is a connected component of a certain complex
Shimura variety. To describe it let $\mathbf{\tilde{G}}$ be the group
of unitary similtudes of $(V,(,))$
\begin{equation}
\mathbf{\tilde{G}=GU}(V,(,)).
\end{equation}
Let $\tilde{K}_{f}^{0}\subset\mathbf{\tilde{G}}(\Bbb{A}_{f})$ be
the subgroup stabilizing $\widehat{L}=L\otimes\widehat{\Bbb{Z}}$
and $\tilde{K}_{f}\subset\tilde{K}_{f}^{0}$ the subgroup of elements
which furthermore induce the identity on $L/NL.$ Let $\tilde{K}_{\infty}$
be the stabilizer of $x_{0}=(\delta/2,0)$ in $\tilde{G}_{\infty}=\mathbf{\tilde{G}}(\Bbb{R}),$
and $\tilde{K}=\tilde{K}_{f}\tilde{K}_{\infty}.$ We have the identification
\begin{equation}
\tilde{G}_{\infty}/\tilde{K}_{\infty}=G_{\infty}/K_{\infty}=\frak{X.}
\end{equation}
The Shimura variety $Sh_{\tilde{K}}$ is a complex quasi-projective
variety identified (as a complex manifold) with the double coset space
\begin{eqnarray}
Sh_{\tilde{K}}(\Bbb{C}) & \simeq & \mathbf{\tilde{G}}(\Bbb{Q})\backslash\mathbf{\tilde{G}}(\Bbb{A})/\tilde{K}\notag\\
 & = & \mathbf{\tilde{G}}(\Bbb{Q})\backslash(\frak{X}\times\mathbf{\tilde{G}}(\Bbb{A}_{f})/\tilde{K}_{f}).
\end{eqnarray}
See {[}8{]},{[}28{]}, or the survey paper {[}13{]}. The connected
components of $Sh_{\tilde{K}}$ are of the form $\Gamma_{j}\backslash\frak{X}$
where the $\Gamma_{j}$ are discrete and torsion-free, and one of
the $\Gamma_{j}$ can be taken to be the $\Gamma$ which we have fixed
above.

By the general theory of Shimura varieties, $Sh_{\tilde{K}}$ admits
a \emph{canonical model} over its reflex field. In our case, the reflex
field turns out to be the field $\mathcal{K},$ and we denote the
canonical model by $S_{\mathcal{K}}.$ The irreducible components
of $Sh_{\tilde{K}}$ are not defined over $\mathcal{K},$ but only
over the ray class field $\mathcal{K}_{N}$ of conductor $N$ over
$\mathcal{K}.$ More precisely, each irreducible component of $S_{\mathcal{K}_{N}}=S_{\mathcal{K}}\times_{\mathcal{K}}\mathcal{K}_{N}$
is already geometrically irreducible. This follows, for example, from
the description of $\pi_{0}(Sh_{\tilde{K}})$ and the Galois action
on it given in {[}8{]}.

\subsection{The Picard modular surface over the ring $R_{0}$ and its arithmetic
compactification}

\subsubsection{The moduli problem}

The canonical model $S_{\mathcal{K}}$ of $Sh_{\tilde{K}}$ has an
interpretation as a fine moduli scheme classifying quadruples $(A,\lambda,\iota,\alpha)_{/R}$
where $R$ is a $\mathcal{K}$-algebra and:
\begin{itemize}
\item $A$ is an abelian three-fold over $R$,
\item $\lambda:A\rightarrow A^{t}$ is a principal polarization,
\item $\iota:\mathcal{O}_{\mathcal{K}}\hookrightarrow End(A/R)$ is an embedding
on which the Rosati involution induced by $\lambda$ is given by $\iota(a)\mapsto\iota(\bar{a}),$
and which makes $Lie(A)$ into a locally free $R$-module of type
$(2,1),$
\item $\alpha$ is a level-$N$ structure. 
\end{itemize}
For a precise definition of what one means by a ``level $N$ structure''
and a full discussion see {[}23{]}, and also the earlier references
{[}25{]}, {[}26{]}, {[}3{]} and {[}13{]}. We write $\mathcal{M}(R)$
for the set of isomorphism classes of such quadruples. Thus $S_{\mathcal{K}}$
represents the functor $\mathcal{M}(-)$ from the category of $\mathcal{K}$-algebras
to the category of sets.

The moduli problem $\mathcal{M}$ makes sense, and is representable,
already over $R_{0}.$

\begin{definition} Let $S$ be the fine moduli scheme representing
$\mathcal{M}$ over $R_{0}.$ \end{definition}

The scheme $S$ is smooth of relative dimension 2 over $R_{0}$, and
$S_{\mathcal{K}}$ is its generic fiber. This allows us to consider
the reduction of $S$ modulo primes of $R_{0}$ (i.e. primes of $\mathcal{K}$
not dividing $2d_{\mathcal{K}}N$).

\subsubsection{The arithmetic compactification}

Larsen's thesis was the first source to work out the smooth compactification
$\bar{S}$ of $S$ over $R_{0}.$ See {[}25{]}, with complements and
corrections in {[}3{]},{[}4{]} and {[}23{]}. Let $R_{N}$ be the integral
closure of $R_{0}$ in the ray class field $\mathcal{K}_{N}.$ Then
the irreducible components of $\bar{S}_{R_{N}}-S_{R_{N}}$ are already
geometrically irreducible, and are elliptic curves (over $R_{N}$)
with complex multiplication by $\mathcal{O}_{\mathcal{K}}$.

\subsubsection{The universal semi-abelian variety and the automorphic vector bundles}

By its very definition as a fine moduli scheme, the scheme $S$ carries
a universal abelian three-fold $\mathcal{A}/S$ (with an additional
structure given by $\lambda,\iota$ and $\alpha$). This $\mathcal{A}$
extends to a semi-abelian scheme (still denoted $\mathcal{A})$ over
$\bar{S}.$ In fact, Larsen and Bella\"{i}che give the boundary $C=\bar{S}-S$
an interpretation as a moduli space of certain semi-abelian schemes
with additional structure, and $\mathcal{A}/C$ is the universal object
arising from this interpretation. (The construction of $\mathcal{A}$
over the whole of $\bar{S}$ requires, of course, a little more effort
than its construction over $S$ and $C$ separately.) The abelian
part of $\mathcal{A}/C$ is an elliptic curve $B$ with complex multiplication
by $\mathcal{O}_{\mathcal{K}}$ and CM type $\Sigma.$ Its toric part
is a torus of the form $\frak{a}\otimes\Bbb{G}_{m}.$ Both $B$ and
the invertible $\mathcal{O}_{\mathcal{K}}$-module $\frak{a}$ are
locally constant along $C$, but the extension class of $B$ by $\frak{a}\otimes\Bbb{G}_{m}$
varies continuously. For details, see {[}3{]}.

What is important for us is that the relative cotangent space at the
origin 
\begin{equation}
\omega_{\mathcal{A}/\bar{S}}=e^{*}\Omega_{\mathcal{A}/\bar{S}}^{1}
\end{equation}
($e:\bar{S}\rightarrow\mathcal{A}$ being the zero section) is a locally
free sheaf on $\bar{S}$ endowed with an $\mathcal{O}_{\mathcal{K}}$
action. It therefore breaks into a direct sum 
\begin{equation}
\omega_{\mathcal{A}/\bar{S}}=\mathcal{P}\oplus\mathcal{L}
\end{equation}
where $\mathcal{P}=\omega_{\mathcal{A}/\bar{S}}(\Sigma)$ is a plane
bundle and $\mathcal{L}=\omega_{\mathcal{A}/\bar{S}}(\bar{\Sigma})$
is a line bundle. See (\ref{type}) for the notation. Here we use
the signature assumption in the moduli problem.

Along the boundary $C$ of $S$ the semi-abelian structure on $\mathcal{A}/C$
provides a filtration 
\begin{equation}
0\rightarrow\mathcal{P}_{0}\rightarrow\mathcal{P}\rightarrow\mathcal{P}_{\mu}\rightarrow0
\end{equation}
where $\mathcal{P}_{0}=\omega_{B/C}$ is the cotangent space at the
origin of the abelian part of $\mathcal{A}$ and $\mathcal{P}_{\mu}$
the $\Sigma$-component of the cotangent space of the toric part.
The line bundles $\mathcal{P}_{0},\mathcal{P}_{\mu}$ and $\mathcal{L}$
are trivial along $C,$ but the extension $\mathcal{P}$ is non-trivial
there.

The vector bundles $\mathcal{P}$ and $\mathcal{L}$ are the basic
\emph{automorphic vector bundles} on $\bar{S}.$ (Vector valued) \emph{modular
forms} are global sections of vector bundles belonging to the tensor
algebra generated by them. For example, for any $k\ge0$ and any $R_{0}$-algebra
$R,$ the space of level-$N$ weight-$k$ (scalar valued) modular
forms over $R$ is 
\begin{equation}
M_{k}(N,R)=H^{0}(\bar{S}\times_{R_{0}}R,\mathcal{L}^{k}).
\end{equation}
By the Köcher principle, this is the same as $H^{0}(S\times_{R_{0}}R,\mathcal{L}^{k}).$

\subsection{Preliminary results on $\mathcal{P}$ and $\mathcal{L}$}

We review some results on the automorphic vector bundles $\mathcal{P}$
and $\mathcal{L}.$ Relations which are special to characteristic
$p$ will be treated in the next section.

\subsubsection{The relation $\det\mathcal{P}\simeq\mathcal{L}$}

Although ``well-known'' and crucial for every purpose, we have not
been able to find an adequate reference for this relation. See also
the remarks following the proof of the proposition.

\begin{proposition} \label{1.3}One has $\det\mathcal{P}\simeq\mathcal{L}$
over $\bar{S}_{\mathcal{K}}.$ \end{proposition}

\begin{proof} Since $Pic(\bar{S}_{\mathcal{K}})\subset Pic(\bar{S}_{\Bbb{C}})$
it is enough to prove the claim over $\Bbb{C}.$ By GAGA, it is enough
to establish the triviality of $\det\mathcal{P}\otimes\mathcal{L}^{-1}$
in the analytic category. We do it over the connected component $\bar{X}_{\Gamma}.$
(The argument for any other connected component is the same.) The
formulae in {[}30, Section 1{]} show that both $\det\mathcal{P}$
and $\mathcal{L}$ can be trivialized over $\frak{X}$ so that the
resulting factor of automorphy is the same, namely 
\begin{equation}
j(\gamma;z,u)=a_{3}z+b_{3}u+c_{3},
\end{equation}
where $(a_{3},b_{3},c_{3})$ is the bottom row of the matrix $\gamma\in\Gamma\subset G_{\infty}.$
If $\sigma$ and $\tau$ are the trivializing sections constructed
by Shimura, the section $\sigma\otimes\tau^{-1}$ trivializes $\det\mathcal{P\otimes L}^{-1}$
and descends to $X_{\Gamma}.$ But using {[}7{]}, for example, one
can easily verify that $\sigma\otimes\tau^{-1}$ extends to a nowhere
vanishing section in an open (classical) neighborhood of any component
of $\bar{X}_{\Gamma}-X_{\Gamma}.$ Thus $\sigma\otimes\tau^{-1}$
trivializes $\det\mathcal{P}\otimes\mathcal{L}^{-1}$ on the whole
of $\bar{X}_{\Gamma}.$

An alternative proof is to use Theorem 4.8 of {[}17{]}. In our case
this theorem gives a \emph{functor} $\mathcal{V}\mapsto[\mathcal{V]}$
from the category of $\mathbf{\tilde{G}}(\Bbb{C})$-equivariant vector
bundles on the compact dual $\Bbb{P}_{\Bbb{C}}^{2}$ of $Sh_{\tilde{K}}$
to the category of vector bundles with $\mathbf{\tilde{G}}(\Bbb{A}_{f})$-action
on the inverse system of Shimura varieties $Sh_{\tilde{K}}.$ Here
$\Bbb{P}_{\Bbb{C}}^{2}=\mathbf{\tilde{G}}(\Bbb{C})/\mathbf{\tilde{H}}(\Bbb{C}),$
where $\mathbf{\tilde{H}}(\Bbb{C})$ is the parabolic group stabilizing
the line $\Bbb{C\cdot}$ $^{t}(\delta/2,0,1)$ in $\mathbf{\tilde{G}}(\Bbb{C})=GL_{3}(\Bbb{C})\times\Bbb{C}^{\times},$
and the irreducible $\mathcal{V}$ are associated with highest weight
representations of the Levi factor $\mathbf{\tilde{L}}(\Bbb{C})$
of $\mathbf{\tilde{H}}(\Bbb{C}).$ It is straight forward to check
that $\det\mathcal{P}$ and $\mathcal{L}$ are associated with the
same character of $\mathbf{\tilde{L}}(\Bbb{C}),$ up to a twist by
a character of $\mathbf{\tilde{G}}(\Bbb{C})$, which affects the $\mathbf{\tilde{G}}(\Bbb{A}_{f})$-action
(hence the normalization of Hecke operators), but not the structure
of the line bundles themselves. The functoriality of Harris' construction
implies that $\det\mathcal{P}$ and $\mathcal{L}$ are isomorphic
also algebraically. \end{proof}

We do not know if $\det\mathcal{P}$ and $\mathcal{L}$ are isomorphic
as algebraic line bundles on $\bar{S}$ (over $R_{0}$). This would
be equivalent to the statement that for every PEL structure $(A,\lambda,\iota,\alpha)\in\mathcal{M}(R),$
for any $R_{0}$-algebra $R,$ $\det(H_{dR}^{1}(A/R)(\Sigma))$ is
the trivial line bundle on $Spec(R)$. In fact, the Hodge filtration
gives a short exact sequence of locally free $R$-modules 
\begin{equation}
0\rightarrow H^{0}(A,\Omega_{A/R}^{1})(\Sigma)\rightarrow H_{dR}^{1}(A/R)(\Sigma)\rightarrow H^{1}(A,\mathcal{O})(\Sigma)\rightarrow0.
\end{equation}
This, together with the canonical isomorphisms 
\begin{eqnarray}
H^{0}(A,\Omega_{A/R}^{1})(\Sigma) & \simeq & \omega_{A/R}(\Sigma)=\mathcal{P}\notag\\
H^{1}(A,\mathcal{O})(\Sigma) & \simeq & Lie(A^{t}/R)(\Sigma)\simeq\omega_{A^{t}/R}^{\vee}(\Sigma)\simeq\omega_{A/R}^{\vee}(\overline{\Sigma})=\mathcal{L}^{\vee}
\end{eqnarray}
(where the last isomorphism in the second formula is induced by the
polarization $\lambda$), yield an isomorphism of line bundles over
$Spec(R)$
\begin{equation}
\det(H_{dR}^{1}(A/R)(\Sigma))\simeq\det\mathcal{P}\otimes\mathcal{L}^{-1}.
\end{equation}

To our regret, we have not been able to establish that this line bundle
is always trivial, although a similar statement in the ``Siegel case'',
namely that for any principally polarized abelian scheme $(A,\lambda)$
over $R,$ $\det H_{dR}^{1}(A/R)$ is trivial, follows at once from
the Hodge filtration. A related remark is that if we take $N=1,$
$\bar{S}$ still makes sense as a stack, but $\det\mathcal{P}\otimes\mathcal{L}^{-1}$
is not expected to be trivial anymore, only torsion. The proposition,
however, suffices to guarantee the following corollary, which is all
that we will be using in the sequel.

\begin{corollary} For any closed point $Spec(k)\rightarrow Spec(R_{0}),$
we have $\det\mathcal{P}\simeq\mathcal{L}$ on $\bar{S}_{k}$. \end{corollary}

\begin{proof} Since $\bar{S}$ is a regular scheme, Proposition \ref{1.3}
implies that $\det\mathcal{P}\otimes\mathcal{L}^{-1}\simeq\mathcal{O}(D)$
for a Weil divisor $D$ which is a $\Bbb{Z}$-linear combination of
irreducible components of vertical fibers over $R_{0}.$ If $Z$ is
an irreducible component of the special fiber $\bar{S}_{k}$, we can
modify $D$ by a multiple of the principle divisor $(p)$ so that
$D$ and $Z$ become disjoint, showing that $\det\mathcal{P}\otimes\mathcal{L}^{-1}|_{Z}$
is trivial. \end{proof}

\subsubsection{The Gauss-Manin connection and the Kodaira-Spencer isomorphism }

Let $\pi:\mathcal{A}\rightarrow S$ be the structure morphism of the
universal abelian scheme over $S$. The Gauss-Manin connection {[}20{]}
\begin{equation}
\nabla:H_{dR}^{1}(\mathcal{A}/S)\rightarrow H_{dR}^{1}(\mathcal{A}/S)\otimes_{\mathcal{O}_{S}}\Omega_{S}^{1}
\end{equation}
(we write $\Omega_{S}^{1}$ for $\Omega_{S/R_{0}}^{1}$) defines the
Kodaira-Spencer map 
\begin{equation}
KS:\omega_{\mathcal{A}/S}\rightarrow\omega_{\mathcal{A}^{t}/S}^{\vee}\otimes_{\mathcal{O}_{S}}\Omega_{S}^{1}.
\end{equation}
Recall that $KS$ is defined by first embedding $\omega_{\mathcal{A}/S}\simeq R^{0}\pi_{*}\Omega_{\mathcal{A}/S}^{1}$
in $H_{dR}^{1}(\mathcal{A}/S),$ then following it by $\nabla,$ and
finally using the projection of $H_{dR}^{1}(\mathcal{A}/S)$ to $R^{1}\pi_{*}\mathcal{O}_{\mathcal{A}}.$
The latter is the relative Lie algebra of $\mathcal{A}^{t}/S,$ hence
may be identified with $\omega_{\mathcal{A}^{t}/S}^{\vee}.$ Unlike
$\nabla,$ $KS$ is $\mathcal{O}_{S}$-linear. It also commutes with
the endomorphisms coming from $\iota(\mathcal{O}_{\mathcal{K}}),$
so defines maps

\begin{eqnarray}
KS(\Sigma) & : & \omega_{\mathcal{A}/S}(\Sigma)\rightarrow\omega_{\mathcal{A}^{t}/S}^{\vee}(\Sigma)\otimes_{\mathcal{O}_{S}}\Omega_{S}^{1}\notag\\
KS(\bar{\Sigma}) & : & \omega_{\mathcal{A}/S}(\bar{\Sigma})\rightarrow\omega_{\mathcal{A}^{t}/S}^{\vee}(\bar{\Sigma})\otimes_{\mathcal{O}_{S}}\Omega_{S}^{1}.
\end{eqnarray}
Alternatively, these are pairings, denoted by the same symbols, 
\begin{eqnarray}
KS(\Sigma) & : & \omega_{\mathcal{A}/S}(\Sigma)\otimes_{\mathcal{O}_{S}}\omega_{\mathcal{A}^{t}/S}(\Sigma)\rightarrow\Omega_{S}^{1}\notag\\
KS(\bar{\Sigma}) & : & \omega_{\mathcal{A}/S}(\bar{\Sigma})\otimes_{\mathcal{O}_{S}}\omega_{\mathcal{A}^{t}/S}(\bar{\Sigma})\rightarrow\Omega_{S}^{1}.\label{KS}
\end{eqnarray}
Observe that (\ref{KS}) are maps between vector bundles of rank 2.

\begin{lemma} The map 
\begin{equation}
KS(\Sigma):\omega_{\mathcal{A}/S}(\Sigma)\otimes_{\mathcal{O}_{S}}\omega_{\mathcal{A}^{t}/S}(\Sigma)\rightarrow\Omega_{S}^{1}
\end{equation}
is an isomorphism, and so is $KS(\bar{\Sigma}).$ \end{lemma}

\begin{proof} This follows from deformation theory. For a self-contained
proof in the arithmetic case, see {[}3{]}, Prop. II.2.1.5. \end{proof}

The type-reversing isomorphism $\lambda^{*}:\omega_{\mathcal{A}^{t}/S}\rightarrow\omega_{\mathcal{A}/S}$
induced by the principal polarization is an isomorphism 
\begin{eqnarray}
\omega_{\mathcal{A}^{t}/S}(\Sigma) & \simeq & \omega_{\mathcal{A}/S}(\bar{\Sigma})=\mathcal{L}\\
\omega_{\mathcal{A}^{t}/S}(\bar{\Sigma}) & \simeq & \omega_{\mathcal{A}/S}(\Sigma)=\mathcal{P}
\end{eqnarray}
and satisfies the \emph{symmetry relation} 
\begin{equation}
KS(\Sigma)(\lambda^{*}x\otimes y)=KS(\bar{\Sigma})(\lambda^{*}y\otimes x)
\end{equation}
for all $x\in\omega_{\mathcal{A}^{t}/S}(\bar{\Sigma})$ and $y\in\omega_{\mathcal{A}^{t}/S}(\Sigma).$
See {[}11{]}, Prop. 9.1 on p.81 (in the Siegel modular case).

\begin{proposition} The Kodaira-Spencer map induces a canonical isomorphism
of vector bundles over $S$
\begin{equation}
\mathcal{P}\otimes\mathcal{L}\simeq\Omega_{S}^{1}.
\end{equation}
\end{proposition}

\begin{proof} Use $\lambda^{*}$ to identify $\omega_{\mathcal{A}^{t}/S}(\Sigma)$
with $\omega_{\mathcal{A}/S}(\bar{\Sigma}).$ \end{proof}

\begin{corollary} Up to a twist by a fractional ideal of $R$, there
is an isomorphism of line bundles $\mathcal{L}^{3}\simeq\Omega_{S}^{2}.$
\end{corollary}

\begin{proof} Take determinants and use $\det(\mathcal{P}\otimes\mathcal{L})\simeq(\det\mathcal{P)}\otimes\mathcal{L}^{2}\simeq\mathcal{L}^{3}.$
Note that we know $\det\mathcal{P}\simeq\mathcal{L}$ only up to a
twist by a fractional ideal of $R_{0}$ (cf. the discussion following
Proposition \ref{1.3}). \end{proof}

\subsubsection{More identities over $\bar{S}$}

We have seen that $\Omega_{S}^{2}\simeq\mathcal{L}^{3}.$ For the
following proposition, compare {[}3{]}, Lemme II.2.1.7.

\begin{proposition} \label{canonical class}Let $C=\bar{S}-S.$ There
is an isomorphism 
\begin{equation}
\Omega_{\bar{S}}^{2}\simeq\mathcal{L}^{3}\otimes\mathcal{O}(C)^{\vee}.
\end{equation}
\end{proposition}

\begin{proof} We base change to an algebraically closed field, so
that $C$ becomes the disjoint union of elliptic curves $E_{i}$ $(1\le i\le h).$
By {[}18{]} II.6.5, $\Omega_{\bar{S}}^{2}\simeq\mathcal{L}^{3}\otimes\bigotimes_{j=1}^{h}\mathcal{O}(E_{j})^{n_{j}}$
for some integers $n_{j}$, and we want to show that $n_{j}=-1$ for
all $j.$ By the adjunction formula on the smooth surface $\bar{S},$
if we denote by $K_{\bar{S}}$ a canonical divisor, $\mathcal{O}(K_{\bar{S}})=\Omega_{\bar{S}}^{2},$
then 
\begin{equation}
0=2g_{E_{j}}-2=E_{j}.(E_{j}+K_{\bar{S}}).
\end{equation}
We conclude that 
\begin{equation}
\deg(\Omega_{\bar{S}}^{2}|_{E_{j}})=E_{j}.K_{\bar{S}}=-E_{j}.E_{j}>0.
\end{equation}
Here $E_{j}.E_{j}<0$ because $E_{j}$ can be contracted to a point
(Grauert's theorem). As $\mathcal{L}|_{E_{j}}$ and $\mathcal{O}(E_{i})|_{E_{j}}$
$(i\neq j)$ are trivial we get 
\begin{equation}
-E_{j}.E_{j}=n_{j}E_{j}.E_{j},
\end{equation}
hence $n_{j}=-1$ as desired. \end{proof}

\section{The Picard modular surface modulo an inert prime}

\subsection{The stratification}

\subsubsection{The three strata}

Let $p$ be a rational prime which is inert in $\mathcal{K}$ and
relatively prime to $2N.$ Then $\kappa_{0}=R_{0}/pR_{0}$ is isomorphic
to $\Bbb{F}_{p^{2}}.$ We fix an algebraic closure $\kappa$ of $\kappa_{0}$
and consider the characteristic $p$ fiber 
\begin{equation}
\bar{S}_{\kappa}=\bar{S}\times_{R_{0}}\kappa.
\end{equation}
\emph{Unless otherwise specified}, in this section we let $S$ and
$\bar{S}$ denote the characteristic $p$ fibers $S_{\kappa}$ and
$\bar{S}_{\kappa}.$ We also use the abbreviation $\omega_{\mathcal{A}}$
for $\omega_{\mathcal{A}/\bar{S}}$ etc.

Recall that an abelian variety over an algebraically closed field
of characteristic $p$ is called \emph{supersingular} if the Newton
polygon of its $p$-divisible group has a constant slope $1/2.$ It
is called \emph{superspecial }if it is isomorphic to a product of
supersingular elliptic curves. The following theorem combines various
results proved in {[}6{]}, {[}31{]} and {[}32{]}.

\begin{theorem} \label{Vollaard}(i) There exists a closed reduced
$1$-dimensional subscheme $S_{ss}\subset\bar{S}$ (the supersingular
locus), disjoint from the cuspidal divisor (i.e. contained in $S$),
which is uniquely characterized by the fact that for any geometric
point $x$ of $S,$ the abelian variety $\mathcal{A}_{x}$ is supersingular
if and only if $x$ lies on $S_{ss}$. The scheme $S_{ss}$ is defined
over $\kappa_{0}.$

(ii) Let $S_{ssp}$ be the singular locus of $S_{ss}.$ Then $x$
lies in $S_{ssp}$ if and only if $\mathcal{A}_{x}$ is superspecial.
If $x\in S_{ssp}$ then 
\begin{equation}
\widehat{\mathcal{O}}_{S_{ss},x}\simeq\Bbb{\kappa}[[u,v]]/(u^{p+1}+v^{p+1}).
\end{equation}

(iii) Assume that $N$ is large enough (depending on $p$). Then the
irreducible components of $S_{ss}$ are nonsingular, and in fact are
all isomorphic to the Fermat curve $\mathcal{C}_{p}$ given by the
equation 
\begin{equation}
x^{p+1}+y^{p+1}+z^{p+1}=0.
\end{equation}
There are $p^{3}+1$ points of $S_{ssp}$ on each irreducible component,
and through each such point pass $p+1$ irreducible components. Any
two irreducible components are either disjoint or intersect transversally
at a unique point.

(iv) Without the assumption of $N$ being large (but under $N\ge3$
as usual) the irreducible components of $S_{ss}$ may have multiple
intersections with each other, including self-intersections. Their
normalizations are nevertheless still isomorphic to $\mathcal{C}_{p}.$
\end{theorem}

\begin{proof} Points (i) and (ii) follow from {[}32 (7.3)(b){]} ($d=3$)
and {[}31, Theorem 3{]}. The structure of $\widehat{\mathcal{O}}_{S_{ss},x}$
is also proved there, but we shall recover it in (\ref{ssp equation})
below.

Point (iii) is {[}31, Theorem 4{]}. Let $\mathcal{T}$ be the Bruhat-Tits
building of the group $G_{p}.$ This is a biregular tree, whose even
vertices are of degree $p^{3}+1,$ and whose odd vertices are of degree
$p+1.$ Fix a connected component $Z$ of $S_{ss}.$ Vollaard identifies
the incidence graph, whose even vertices correspond to the irreducible
components of $Z,$ and whose odd vertices correspond to the points
of $Z\cap S_{ssp}$ (edges denoting incidence relations) with the
quotient of $\mathcal{T}$ by a certain discrete cocompact arithmetic
subgroup of $G_{p}$. The condition on $N$ being large is necessary
to guarantee that the action of that group on $\mathcal{T}$ is ``good''.

Point (iv) is not explicitly stated in {[}31{]}, but follows from
the discussion there. Since this point is used later on, we explain
it here. We refer to the notation of {[}31, Section 6{]}. Our ``$N$
large'' condition is Vollaard's ``$C^{p}$ small''. This is used
by her in two ways. First, it is needed to guarantee that the groups
\begin{equation}
\Gamma_{j}=I(\Bbb{Q})\cap g_{j}^{-1}C^{p}g_{j}
\end{equation}
are torsion free. For this, it is enough to have $N\ge3,$ by Serre's
lemma. Second, Vollaard needs the assumption ``$C^{p}$ small''
in the proof of Theorem 6.1, to guarantee that the above-mentioned
action of $\Gamma_{j}$ on the tree $\mathcal{T}$ is good. For that,
she needs the distances $u(\Lambda,g\Lambda)$ to be at least 6 for
every $1\neq g\in\Gamma_{j}^{0}$ and $\Lambda\in\mathcal{L}_{0}.$
This measure of smallness translates to $N\ge N_{0}(p)$ where $N_{0}(p)$
depends on $p.$

If we only assume $N\ge3,$ then Theorem 6.1 of {[}31{]} does not
hold. However, $\Gamma_{j}$ still acts freely on the set of irreducible
components of $\mathcal{N}^{red}.$ Indeed, the stabilizer of any
given irreducible component is a finite subgroup of $\Gamma_{j}$,
hence trivial. It follows that while the irreducible components of
$\mathcal{N}^{red},$ which are all isomorphic to $\mathcal{C}_{p},$
may acquire self-intersections in $\Gamma_{j}\backslash\mathcal{N}^{red},$
their normalizations will still be $\mathcal{C}_{p}.$ \end{proof}

We call $\bar{S}_{\mu}=\bar{S}-S_{ss}$ (or $S_{\mu}=\bar{S}_{\mu}\cap S$)
the $\mu$\emph{-ordinary} or \emph{generic} locus, $S_{gss}=S_{ss}-S_{ssp}$
the \emph{general supersingular locus}, and $S_{ssp}$ the \emph{superspecial}
\emph{locus}. Then $\bar{S}=\bar{S}_{\mu}\cup S_{gss}\cup S_{ssp}$
is a stratification: the three strata are of pure dimension 2,1, and
0 respectively, the closure of each stratum contains the lower dimensional
ones, and each of the three is open in its closure.

\subsubsection{The $p$-divisible group}

Let $x:Spec(k)\rightarrow S$ ($k$ an algebraically closed field)
be a geometric point of $S,$ $\mathcal{A}_{x}$ the corresponding
fiber of $\mathcal{A}$, and $\mathcal{A}_{x}(p)$ its $p$-divisible
group. Let $\frak{G}$ be the $p$-divisible group of a supersingular
elliptic curve over $k$ (the group denoted by $G_{1,1}$ in the Manin-Dieudonné
classification). The following theorem can be deduced from {[}6{]}
and {[}31{]}.

\begin{theorem} (i) If $x\in S_{\mu}$ then 
\begin{equation}
\mathcal{A}_{x}(p)\simeq(\mathcal{O}_{\mathcal{K}}\otimes\mu_{p^{\infty}})\times\frak{G}\times(\mathcal{O}_{\mathcal{K}}\otimes\Bbb{Q}_{p}/\Bbb{Z}_{p}).
\end{equation}

(ii) If $x\in S_{ss}$ then $\mathcal{A}_{x}(p)$ is isogenous to
$\frak{G}^{3},$ and $x\in S_{ssp}$ if and only if the two groups
are isomorphic. \end{theorem}

While the $p$-divisible group of a $\mu$-ordinary geometric fiber
actually splits as a product of its multiplicative, local-local and
étale parts, over the whole of $S_{\mu}$ we only get a filtration
\begin{equation}
0\subset Fil^{2}\mathcal{A}(p)\subset Fil^{1}\mathcal{A}(p)\subset Fil^{0}\mathcal{A}(p)=\mathcal{A}(p)
\end{equation}
by $\mathcal{O}_{\mathcal{K}}$-stable $p$-divisible groups. Here
$gr^{2}=Fil^{2}$ is of multiplicative type, $gr^{1}=Fil^{1}/Fil^{2}$
is a local-local group and $gr^{0}=Fil^{0}/Fil^{1}$ is étale, each
of height $2$ ($\mathcal{O}_{\mathcal{K}}$-height 1).

\subsection{New relations between $\mathcal{P}$ and $\mathcal{L}$ in characteristic
$p$}

\subsubsection{The line bundles $\mathcal{P}_{0}$ and $\mathcal{P}_{\mu}$ over
$\bar{S}_{\mu}$}

Consider the universal semi-abelian variety $\mathcal{A}$ over the
Zariski open set $\bar{S}_{\mu}.$ Over the cuspidal divisor $C=\bar{S}-S,$
where $\mathcal{A}$ becomes an extension of a supersingular elliptic
curve $B$ (with $\mathcal{O}_{\mathcal{K}}$-signature $(1,0)$)
by a 2-dimensional torus (with$\,\mathcal{O}_{\mathcal{K}}$-signature
$(1,1)$), $\mathcal{P}=\omega_{\mathcal{A}}(\Sigma)$ admits a canonical
filtration 
\begin{equation}
0\rightarrow\mathcal{P}_{0}\rightarrow\mathcal{P}\rightarrow\mathcal{P}_{\mu}\rightarrow0.\label{Fil on P}
\end{equation}
Here $\mathcal{P}_{0}$ is the cotangent space to the abelian part
$B$ of $\mathcal{A},$ and $\mathcal{P}_{\mu}$ is the $\Sigma$-component
of the cotangent space to the toric part of $\mathcal{A}.$ This filtration
exists already in characteristic 0, but when we reduce the Picard
surface modulo $p$ it extends, as we now show, to the whole of $\bar{S}_{\mu}.$

Let $\mathcal{A}[p]^{0}$ be the connected part of the subgroup scheme
$\mathcal{A}[p]$ over $\bar{S}_{\mu}.$ Then $\mathcal{A}[p]^{0}$
is finite flat of rank $p^{4}.$ (It is clearly flat and quasi-finite,
and the fiber rank can be computed separately on $C$ and on $S_{\mu}.$
Since the rank is constant, the morphism to $\bar{S}_{\mu}$ is actually
finite, \emph{cf.} {[}9{]}, Lemme 1.19.) Let 
\begin{equation}
0\subset\mathcal{A}[p]^{\mu}\subset\mathcal{A}[p]^{0}
\end{equation}
be the maximal subgroup-scheme of multiplicative type. Since at every
closed point of $\bar{S}_{\mu}$, $\mathcal{A}[p]^{\mu}$ is of rank
$p^{2},$ this subgroup is also finite flat over $\bar{S}_{\mu}.$
It is also $\mathcal{O}_{\mathcal{K}}$-invariant. Over the cuspidal
divisor $C$, $\mathcal{A}[p]^{\mu}$ is the $p$-torsion in the toric
part of $\mathcal{A},$ and over $S_{\mu}$
\begin{equation}
\mathcal{A}[p]^{\mu}=\mathcal{A}[p]\cap Fil^{2}\mathcal{A}(p).
\end{equation}

As $\omega_{\mathcal{A}}$ is killed by $p,$ we have $\omega_{\mathcal{A}}=\omega_{\mathcal{A}[p]}=\omega_{\mathcal{A}[p]^{0}}$.
Let $\omega_{\mathcal{A}}^{\mu}=\omega_{\mathcal{A}[p]^{\mu}},$ a
rank-2 $\mathcal{O}_{\mathcal{K}}$-vector bundle of type $(1,1).$
The kernel of $\omega_{\mathcal{A}[p]^{0}}\rightarrow\omega_{\mathcal{A}[p]^{\mu}}$
is then a line bundle $\mathcal{P}_{0}$ of type $(1,0)$ and we get
the short exact sequence 
\begin{equation}
0\rightarrow\mathcal{P}_{0}\rightarrow\omega_{\mathcal{A}}\rightarrow\omega_{\mathcal{A}}^{\mu}\rightarrow0
\end{equation}
over the whole of $\bar{S}_{\mu}.$ Decomposing according to types
and setting $\mathcal{P}_{\mu}=\omega_{\mathcal{A}}^{\mu}(\Sigma),$
we get the desired filtration on $\mathcal{P}$.

\subsubsection{Frobenius and Verschiebung}

Let $\mathcal{A}^{(p)}$ be the base change of $\mathcal{A}$ with
respect to the absolute Frobenius morphism of degree $p$ of $\bar{S}.$
In other words, if we denote by $\phi$ the homomorphism $x\mapsto x^{p}$
(of any $\Bbb{F}_{p}$-algebra), and by $\Phi:\bar{S}\rightarrow\bar{S}$
the corresponding map of schemes (which is \emph{not }$\kappa_{0}$-linear),
then 
\begin{equation}
\mathcal{A}^{(p)}=\mathcal{A}\times_{\bar{S},\Phi}\bar{S}.
\end{equation}
If $\mathcal{M}$ is a coherent sheaf on $\mathcal{A}$, we denote
by $\mathcal{M}^{(p)}=\Phi^{*}\mathcal{M}$ its base-change to $\mathcal{A}^{(p)}.$
If $\mathcal{M}$ is endowed with an $\mathcal{O}_{\mathcal{K}}$-action,
so is, via base-change, $\mathcal{M}^{(p)}.$ However, if $\mathcal{M}$
is a vector bundle of type $(a,b)$ then $\mathcal{M}^{(p)}$ is of
type $(b,a)$, because $x\mapsto x^{p}$ interchanges $\Sigma$ with
$\overline{\Sigma}.$

The \emph{relative Frobenius} is an $\mathcal{O}_{\bar{S}}$-linear
isogeny $Frob_{\mathcal{A}}:\mathcal{A}\rightarrow\mathcal{A}^{(p)},$
characterized by the fact that $pr_{1}\circ Frob_{\mathcal{A}}$ is
the absolute Frobenius morphism of $\mathcal{A}.$ Over $S$ (but
not over the boundary $C$) we have the dual abelian scheme $\mathcal{A}^{t}$,
and the \emph{Verschiebung }$Ver_{\mathcal{A}}:\mathcal{A}^{(p)}\rightarrow\mathcal{A}$
is the $\mathcal{O}_{S}$-linear isogeny which is dual to $Frob_{\mathcal{A}^{t}}:\mathcal{A}^{t}\rightarrow(\mathcal{A}^{t})^{(p)}.$

We clearly have $\omega_{\mathcal{A}^{(p)}}=\omega_{\mathcal{A}}^{(p)}$,
and we let 
\begin{equation}
F:\omega_{\mathcal{A}}^{(p)}\rightarrow\omega_{\mathcal{A}},\,\,\,V:\omega_{\mathcal{A}}\rightarrow\omega_{\mathcal{A}}^{(p)}
\end{equation}
be the $\mathcal{O}_{\bar{S}}$-linear maps of vector bundles induced
by the isogenies $Frob_{\mathcal{A}}$ and $Ver_{\mathcal{A}}$ on
the cotangent spaces. Note, however, that while $F$ is defined everywhere$,$
$V$ is so far defined only over $S.$

To extend the definition of $V$ over $\bar{S}$ we consider the finite
flat group scheme $\mathcal{G}=\mathcal{A}[p]^{0}$ over $\bar{S}_{\mu}$
as in the previous subsection. We now have the $\mathcal{O}_{\bar{S}}$-linear
homomorphism $Ver_{\mathcal{G}}:\mathcal{G}^{(p)}\rightarrow\mathcal{G}$
also over the boundary. Under Cartier duality $Ver_{\mathcal{G}}$
is dual to $Frob_{\mathcal{G}^{D}},$ where we denote by $\mathcal{G}^{D}$
the Cartier dual of $\mathcal{G}.$ Over $S_{\mu}$ it coincides with
the homomorphism induced by $Ver_{\mathcal{A}}.$ This follows at
once from the identification of $\mathcal{A}[p]^{D}$ with $\mathcal{A}^{t}[p]$
(Weil pairing).

Recall that $\omega_{\mathcal{G}}=\omega_{\mathcal{A}},$ and similarly
$\omega_{\mathcal{G}^{(p)}}=\omega_{\mathcal{A}^{(p)}}=\omega_{\mathcal{A}}^{(p)}.$
The morphism $Ver_{\mathcal{G}}:\mathcal{G}^{(p)}\rightarrow\mathcal{G}$
therefore induces a homomorphism of vector bundles over $\bar{S}_{\mu},$
$V:\omega_{\mathcal{A}}\rightarrow\omega_{\mathcal{A}}^{(p)}$, which
extends the previously defined map $V$ to $\bar{S}.$

Taking $\Sigma$-components we get 
\begin{equation}
V:\mathcal{P}=\omega_{\mathcal{A}}(\Sigma)\rightarrow\omega_{\mathcal{A}}^{(p)}(\Sigma)=\omega_{\mathcal{A}}(\bar{\Sigma})^{(p)}=\mathcal{L}^{(p)}.
\end{equation}

Over $\bar{S}_{\mu}$ this map fits in a commutative diagram 
\begin{equation}
\begin{array}{lllllllll}
0 & \leftarrow & \omega_{\mathcal{A}}^{\mu}(\Sigma)=\mathcal{P}_{\mu} & \leftarrow & \mathcal{P} & \leftarrow & \mathcal{P}_{0} & \leftarrow & 0\\
 &  & \downarrow\,\simeq &  & \downarrow V &  & \downarrow\\
0 & \leftarrow & (\omega_{\mathcal{A}}^{\mu}(\bar{\Sigma}))^{(p)} & \leftarrow & \mathcal{L}^{(p)} & \leftarrow & 0 & \leftarrow & 0
\end{array}.
\end{equation}
The right vertical arrow is 0 since $V$ kills $\mathcal{P}_{0},$
as $\mathcal{A}[p]^{0}/\mathcal{A}[p]^{\mu}$ is of local-local type,
hence $Ver_{\mathcal{G}}$ acts on it nilpotently. The left vertical
map is an isomorphism, since $Ver$ is an isomorphism on $p$-divisible
groups of multiplicative type. We conclude that over $\bar{S}_{\mu}$
\begin{equation}
\mathcal{P}_{0}=\ker(V:\mathcal{P}\rightarrow\mathcal{L}^{(p)}).
\end{equation}

\subsubsection{Relations between $\mathcal{P}_{0},\mathcal{P}_{\mu}$ and $\mathcal{L}$
over $\bar{S}_{\mu}$}

We first recall a general lemma.

\begin{lemma} \label{M^p}Let $\mathcal{M}$ be a line bundle over
a scheme $S$ in characteristic $p.$ Let $\Phi:S\rightarrow S$ be
the absolute Frobenius and $\mathcal{M}^{(p)}=\Phi^{*}\mathcal{M}.$
Then the map $\mathcal{M}^{(p)}\rightarrow\mathcal{M}^{p}$
\begin{equation}
a\otimes m\mapsto a\cdot m\otimes\cdots\otimes m
\end{equation}
is an isomorphism of line bundles over $S.$ \end{lemma}

Since $\mathcal{L}^{(p)}\simeq\mathcal{L}^{p}$ by the lemma, we have
\begin{equation}
\mathcal{L}^{p}\simeq\mathcal{P}/\mathcal{P}_{0}=\mathcal{P}_{\mu}.
\end{equation}
Finally, from $\mathcal{P}_{0}\otimes\mathcal{P}_{\mu}\simeq\det\mathcal{P}\simeq\mathcal{L}$
we get 
\begin{equation}
\mathcal{P}_{0}\simeq\mathcal{L}^{1-p}.
\end{equation}
We have proved:

\begin{proposition} Over $\bar{S}_{\mu},$ $\mathcal{P}_{\mu}\simeq\mathcal{L}^{p}$
and $\mathcal{P}_{0}\simeq\mathcal{L}^{1-p}.$ \end{proposition}

In the same vein we get a commutative diagram for the $\bar{\Sigma}$
parts 
\begin{equation}
\begin{array}{lllllllll}
0 & \leftarrow & \omega_{\mathcal{A}}^{\mu}(\bar{\Sigma}) & \leftarrow & \mathcal{L} & \leftarrow & 0 & \leftarrow & 0\\
 &  & \downarrow\,\simeq &  & \downarrow V &  & \downarrow\\
0 & \leftarrow & (\omega_{\mathcal{A}}^{\mu}(\Sigma))^{(p)} & \leftarrow & \mathcal{P}^{(p)} & \leftarrow & \mathcal{P}_{0}^{(p)} & \leftarrow & 0
\end{array}\label{P^p}
\end{equation}
and deduce that $V$ is injective on $\mathcal{L}$ and 
\begin{equation}
\mathcal{P}^{(p)}=\mathcal{P}_{0}^{(p)}\oplus V(\mathcal{L}).
\end{equation}
Thus over $\bar{S}_{\mu},$ $\mathcal{P}$ has a canonical filtration
by $\mathcal{P}_{0},$ but the induced filtration on $\mathcal{P}^{(p)}$
already splits as a direct sum.

\begin{remark} Restricting attention to a connected component $E$
of $C,$ $\mathcal{P}|_{E}$ is a \emph{non-split} extension of $\mathcal{P}_{\mu}$
by $\mathcal{P}_{0}.$ However, both $\mathcal{P}_{\mu}$ and $\mathcal{P}_{0}$
are trivial on $E,$ so the extension is described by a non-zero class
$\xi\in$ $H^{1}(E,\mathcal{O}_{E}).$ The extension $\mathcal{P}^{(p)}|_{E}$
is then described by $\xi^{(p)}.$ The semilinear map $\xi\mapsto\xi^{(p)}$
is the Cartier-Manin operator, and since $E$ is a supersingular elliptic
curve, $\xi^{(p)}=0$ and $\mathcal{P}^{(p)}|_{E}$ splits. Thus at
least over $C,$ the splitting of $\mathcal{P}^{(p)}$ is consistent
with what we already know. \end{remark}

Since $V$ induces an isomorphism of $\mathcal{L}$ onto $\mathcal{P}^{(p)}/\mathcal{P}_{0}^{(p)}\simeq(\mathcal{P}/\mathcal{P}_{0})^{p}$
and $\mathcal{P}/\mathcal{P}_{0}\simeq\mathcal{L}^{p}$ we conclude
that over $\bar{S}_{\mu},$ $\mathcal{L}\simeq\mathcal{L}^{p^{2}}.$
In the next section we realize this isomorphism via the Hasse invariant.
Combining what was proved so far we get the following.

\begin{proposition} Over $\bar{S}_{\mu}$, $\mathcal{L}^{p^{2}}\simeq\mathcal{L}$.
For $k\ge1$ odd, $\mathcal{P}^{(p^{k})}\simeq\mathcal{L}^{p-1}\oplus\mathcal{L}.$
For $k\ge2$ even, $\mathcal{P}^{(p^{k})}\simeq\mathcal{L}^{1-p}\oplus\mathcal{L}^{p},$
but for $k=0$ we only have an exact sequence 
\begin{equation}
0\rightarrow\mathcal{L}^{1-p}\rightarrow\mathcal{P}\rightarrow\mathcal{L}^{p}\rightarrow0.
\end{equation}
\end{proposition}

\begin{corollary} Over $\bar{S}_{\mu},$ $\mathcal{L}^{p^{2}-1},\mathcal{P}_{\mu}^{p^{2}-1}$
and $\mathcal{P}_{0}^{p+1}$ are trivial line bundles. \end{corollary}

\subsubsection{Extending the filtration on $\mathcal{P}$ over $S_{gss}$}

In order to determine to what extent the filtration on $\mathcal{P}$
and the relation between $\mathcal{L}$ and the two graded pieces
of the filtration extend into the supersingular locus, we have to
employ Dieudonné theory.

\begin{proposition} Let $\mathcal{P}_{0}=\ker(V:\mathcal{P}\rightarrow\mathcal{L}^{(p)}).$
Then over the whole of $\bar{S}-S_{ssp},$ $V(\mathcal{P})=\mathcal{L}^{(p)}$
and $\mathcal{P}_{0}$ is a rank 1 submodule. Let $\mathcal{P}_{\mu}=\mathcal{P}/\mathcal{P}_{0}.$
Then $\mathcal{P}_{\mu}\simeq\mathcal{L}^{p},$ $\mathcal{P}_{0}\simeq\mathcal{L}^{1-p}$
and the filtration (\ref{Fil on P}) is valid there. \end{proposition}

\begin{proof} Everything is a formal consequence of the fact that
$V$ maps $\mathcal{P}$ \emph{onto} $\mathcal{L}^{(p)},$ and the
relation $\det\mathcal{P}\simeq\mathcal{L}.$ Over $\bar{S}_{\mu}$
the proposition was verified in the previous subsection, so it is
enough to prove that $V(\mathcal{P})=\mathcal{L}^{(p)}$ in the fiber
of any geometric point $x\in S_{gss}(k)$ ($k$ algebraically closed).
We use the description of $H_{dR}^{1}(\mathcal{A}_{x}/k)$ given in
Lemma \ref{Braid} below, due to Bültel and Wedhorn. In the notation
of that lemma, $\mathcal{P}_{x}$ is spanned over $k$ by $e_{1}$
and $e_{2}$ and $\mathcal{L}_{x}$ by $f_{3},$ while $V(e_{1})=0,$
$V(e_{2})=f_{3}^{(p)}.$ This concludes the proof. \end{proof}

For the next lemma let $D_{0}=H_{dR}^{1}(\mathcal{A}_{x}/k),$ where
$x\in S_{gss}(k)$ and $k$ is algebraically closed. We identify $D_{0}$
with the reduction modulo $p$ of the (contravariant) Dieudonné module
of $\mathcal{A}_{x}.$ It is therefore equipped with $k$-linear maps
$F:D_{0}^{(p)}\rightarrow D_{0}$ and $V:D_{0}\rightarrow D_{0}^{(p)}$
where $D_{0}^{(p)}=k\otimes_{\phi,k}D_{0}$ as usual.

\begin{lemma} \label{Braid}There exists a basis $e_{1},e_{2},f_{3},f_{1},f_{2},e_{3}$
of $D_{0}$ with the following properties. Denote by $e_{1}^{(p)}=1\otimes e_{1}\in D_{0}^{(p)}$
etc.

(i) $\mathcal{O}_{\mathcal{K}}$ acts on the $e_{i}$ via $\Sigma$
and on the $f_{i}$ via $\bar{\Sigma}$ (hence it acts on the $e_{i}^{(p)}$
via $\bar{\Sigma}$ and on the $f_{i}^{(p)}$ via $\Sigma).$

(ii) The symplectic pairing on $D_{0}$ induced by the principal polarization
$\lambda_{x}$ satisfies 
\begin{equation}
\left\langle e_{i},f_{j}\right\rangle =-\left\langle f_{j},e_{i}\right\rangle =\delta_{ij},\,\,\,\left\langle e_{i},e_{j}\right\rangle =\left\langle f_{i},f_{j}\right\rangle =0.
\end{equation}

(iii) The vectors $e_{1},e_{2},f_{3}$ form a basis for the cotangent
space $\omega_{\mathcal{A}}$ at $x.$ Hence $e_{1}$ and $e_{2}$
span $\mathcal{P}$ and $f_{3}$ spans $\mathcal{L}.$

(iv) $\ker(V)$ is spanned by $e_{1},f_{2},e_{3}.$ Hence $\mathcal{P}_{0}=\mathcal{P}\cap\ker(V)$
is spanned by $e_{1}.$

(v) $Ve_{2}=f_{3}^{(p)},\,Vf_{3}=e_{1}^{(p)},\,Vf_{1}=e_{2}^{(p)}.$

(vi) $Ff_{1}^{(p)}=-e_{3},\,Ff_{2}^{(p)}=-e_{1},Fe_{3}^{(p)}=-f_{2}.$
\end{lemma}

\begin{proof} Up to a slight change of notation, this is the unitary
Dieudonné module which Bültel and Wedhorn call a ``braid of length
3'' and denote by $\bar{B}(3),$ cf {[}6{]} (3.2). The classification
in loc. cit. Proposition 3.6 shows that the Dieudonné module of a
$\mu$-ordinary abelian variety is isomorphic to $\bar{B}(2)\oplus\bar{S},$
that of a gss abelian variety is isomorphic to $\bar{B}(3)$ and in
the superspecial case we get $\bar{B}(1)\oplus\bar{S}^{2}.$ \end{proof}

\begin{proposition} Over the whole of $\bar{S}-S_{ssp}$, $V$ maps
$\mathcal{L}$ injectively onto a sub-line-bundle of $\mathcal{P}^{(p)}.$
\end{proposition}

\begin{proof} Once again, we know it already over $\bar{S}_{\mu},$
and it remains to check the assertion fiber-wise on $S_{gss}.$ We
refer again to Lemma \ref{Braid}, and find that $V(f_{3})=e_{1}^{(p)},$
which proves our claim. \end{proof}

The emerging picture is this: Outside the superspecial points, $V$
maps $\mathcal{L}$ injectively onto a sub-line-bundle of $\mathcal{P}^{(p)}$,
and $V^{(p)}$ maps $\mathcal{P}^{(p)}$ surjectively onto $\mathcal{L}^{(p^{2})}.$
However, the line $V(\mathcal{L})$ coincides with the line $\mathcal{P}_{0}^{(p)}=\ker(V^{(p)})$
only on the general supersingular locus, while on its complement $\bar{S}_{\mu}$
the two lines make up a frame for $\mathcal{P}^{(p)}$ (\ref{P^p}).
One can be a little more precise. The equation 
\begin{equation}
V(\mathcal{L})=\mathcal{P}_{0}^{(p)}
\end{equation}
(i.e. $V^{(p)}\circ V(\mathcal{L})=0$) is the \emph{defining equation}
of $S_{ss}$ in the sense that when expressed in local coordinates
it defines $S_{ss}$ with its \emph{reduced} subscheme structure.
See Proposition \ref{Hasse Invariant} below.

\subsubsection{Non-extendibility of the filtration across $S_{ssp}$}

For a superspecial $x,$ $\mathcal{A}_{x}$ is isomorphic to a product
of three supersingular elliptic curves, so $V$ vanishes on the whole
of $\omega_{\mathcal{A}}$ at $x.$ The analysis of the last paragraph
breaks up. To complete the picture, we shall now prove that there
does not exist any way to extend the filtration (\ref{Fil on P})
across such an $x.$

\begin{proposition} \label{superspecial}It is impossible to extend
the filtration $0\rightarrow\mathcal{P}_{0}\rightarrow\mathcal{P}\rightarrow\mathcal{P}_{\mu}\rightarrow0$
along $S_{ss}$ in a neighborhood of a superspecial point $x$. \end{proposition}

\begin{proof} At any superspecial point there are $p+1$ branches
of $S_{ss}$ meeting transversally. We shall prove the proposition
by showing that along any one of these branches (labelled by $\zeta,$
a $p+1$-st root of $-1$) $\mathcal{P}_{0}$ approaches a line $\mathcal{P}_{x}[\zeta]\subset\mathcal{P}_{x},$
but these $p+1$ lines are distinct. In other words, on the normalization
of $S_{ss}$ we can extend the filtration uniquely, but the extension
does not descend to $S_{ss}.$

Before we go into the proof a word of explanation is needed. In order
to study the deformation of the action of $V$ on $\omega_{\mathcal{A}}$
near a \emph{general }supersingular point $x\in S_{gss}$, the first
infinitesimal neighborhood of $x$ suffices, and one ends up using
Grothendieck's crystalline deformation theory (see \ref{PD deformation}
below). At a point $x\in S_{ssp},$ in contrast, we need to work in
the full formal neighborhood of $x$ in $S$, or at least in an Artinian
neighborhood which no longer admits a divided power structure. The
reason is that the singularity of $S_{ss}$ at $x$ is formally of
the type $Spec(\kappa[[u,v]]/(u^{p+1}+v^{p+1}))$. Crystalline deformation
theory is inadequate, and we need to use Zink's ``displays''. As
the theory of displays is covariant, we start with the covariant Cartier
module of $A=\mathcal{A}_{x}$ rather than the contravariant Dieudonné
module, and look for its universal deformation.

Let us review the (confusing) functoriality of these two modules.
For the moment, let $A$ be any abelian variety over $\kappa.$ Here
$\kappa$ can be any perfect field of characteristic $p.$ If $D$
is the (contravariant) Dieudonné module of $A$ and $M$ is its (covariant)
Cartier module, then $D/pD=H_{dR}^{1}(A)$ and $M/pM=H_{dR}^{1}(A^{t})$
are set in duality. The dual of 
\begin{equation}
V:D/pD\rightarrow(D/pD)^{(p)}
\end{equation}
($V=Ver_{A}^{*}$, $Ver_{A}$ being the Verschiebung isogeny from
$A^{(p)}$ to $A$) is the map 
\begin{equation}
F:(M/pM)^{(p)}\rightarrow M/pM
\end{equation}
($F=Frob_{A^{t}}^{*},$ $Frob_{A^{t}}$ being the Frobenius isogeny
from $A^{t}$ to $A^{t(p)}$). As usual, since $\kappa$ is perfect,
we may view $V$ as a $\phi^{-1}$-linear endomorphism of $D/pD$,
and $F$ as a $\phi$-linear endomorphism of $M/pM.$ Replacing $A$
by $A^{t}$ we then also have semi-linear endomorphisms $F$ of $D/pD$
and $V$ of $M/pM.$ The Hodge filtration $\omega_{A}\subset H_{dR}^{1}(A)$
is $(D/pD)[F]$. Its dual is the quotient $Lie(A)=H^{1}(A^{t},\mathcal{O})$
of $H_{dR}^{1}(A^{t}),$ identified with $M/VM.$ Compare {[}29{]},
Corollary 5.11.

This reminder tells us that when we pass from the contravariant theory
to the covariant one, instead of looking for the deformation of $V$
on $\omega_{A}$ we should look for the deformation of $F$ on $Lie(A)=M/VM.$
At a superspecial point $F$ annihilates $Lie(A),$ but at nearby
points in $S$ it need not annihilate it anymore.

Now let $x\in S_{ssp}$ and $A=\mathcal{A}_{x}.$ The Cartier module
(modulo $p$) $M/pM$ of $A$ admits a symplectic basis $f_{3},e_{1},e_{2},e_{3},f_{1},f_{2}$
where $\mathcal{O}_{\mathcal{K}}$ acts on the $e_{i}$ via $\Sigma$
and on the $f_{i}$ via $\bar{\Sigma}$, where the polarization pairing
is $\left\langle e_{i},f_{j}\right\rangle =-\left\langle f_{j},e_{i}\right\rangle =\delta_{ij}$
and $\left\langle e_{i},e_{j}\right\rangle =\left\langle f_{i},f_{j}\right\rangle =0,$
and where $f_{3},e_{1},e_{2}$ project to a basis of $Lie(A)=M/VM.$
With an appropriate choice of the basis, the Frobenius $F$ on $M/pM$
is the $\phi$-linear map whose matrix with respect to the basis $f_{3},e_{1},e_{2},e_{3},f_{1},f_{2}$
is 
\begin{equation}
\left(\begin{array}{llllll}
0 &  &  & 0\\
 & 0 &  &  & 0\\
 &  & 0 &  &  & 0\\
1 &  &  & 0\\
 & 1 &  &  & 0\\
 &  & 1 &  &  & 0
\end{array}\right).
\end{equation}
All this can be deduced from {[}6{]}, 3.2.

To construct the universal display we follow the method of {[}15{]}.
See also {[}1{]}. With local coordinates $u$ and $v$ we write $\widehat{S}=Spf\kappa[[u,v]]$
for the formal completion of $S$ at $x.$ We study the deformation
of $F$ to 
\begin{equation}
F:H_{dR}^{1}(\mathcal{A}^{t}/\widehat{S})^{(p)}\rightarrow H_{dR}^{1}(\mathcal{A}^{t}/\widehat{S}).
\end{equation}
If $A/R$ is any abelian scheme over an $\Bbb{F}_{p}$-algebra $R$
and $a$ any endomorphism of $A/R,$ then $Frob:A\rightarrow A^{(p)}$
satisfies 
\begin{equation}
Frob\circ a=a^{(p)}\circ Frob
\end{equation}
where $a^{(p)}=1\otimes a$ is the base-change of $a$ to $\mathcal{A}^{(p)}.$
Thus $F$ commutes with the $\mathcal{O}_{\mathcal{K}}$-structure
on de Rham cohomology. Note however that $H_{dR}^{1}(\mathcal{A}^{t}/\widehat{S})^{(p)}(\Sigma)=H_{dR}^{1}(\mathcal{A}^{t}/\widehat{S})(\overline{\Sigma})^{(p)}$
and vice versa.

We use a basis $f_{3},e_{1},\dots,f_{2}$ of $H_{dR}^{1}(\mathcal{A}^{t}/\widehat{S})$
satisfying the same assumptions as above with respect to the $\mathcal{O}_{\mathcal{K}}$-type
and the polarization pairing. A bit of elementary algebra, which we
skip, shows that one can modify the local coordinates $u$ and $v,$
\emph{and} the basis of $H_{dR}^{1}(\mathcal{A}^{t}/\widehat{S})$
(keeping our assumptions on the $\mathcal{O}_{\mathcal{K}}$-type
and the polarization), so that the universal Frobenius is given by
the (Hasse-Witt) matrix 
\begin{equation}
F=\left(\begin{array}{llllll}
0 & u & v & 0\\
u & 0 &  &  & 0\\
v &  & 0 &  &  & 0\\
1 &  &  & 0\\
 & 1 &  &  & 0\\
 &  & 1 &  &  & 0
\end{array}\right).
\end{equation}
This means, as usual, that 
\begin{equation}
F(f_{3}^{(p)})=ue_{1}+ve_{2}+e_{3},
\end{equation}
etc. Since the first three vectors ($f_{3},e_{1},e_{2}$) project
onto a basis of $Lie(\mathcal{A})$, the matrix of $F:Lie(\mathcal{A})^{(p)}\rightarrow Lie(\mathcal{A}\Bbb{)}$
is the $3\times3$ upper left block, and the matrix of $F^{2}(=F\circ F^{(p)})$
is (note the semilinearity) 
\begin{equation}
\left(\begin{array}{lll}
u^{p+1}+v^{p+1}\\
 & u^{p+1} & uv^{p}\\
 & vu^{p} & v^{p+1}
\end{array}\right).
\end{equation}
Thus on $Lie(\mathcal{A})(\bar{\Sigma})^{(p^{2})}=\mathcal{L}^{(p^{2})\vee}$
the action of $F^{2}$ is given by multiplication by $u^{p+1}+v^{p+1}$.
As the supersingular locus is the locus where the action of $F$ on
the Lie algebra is nilpotent, we recover the fact that the local (formal)
equation of $S_{ss}$ at $x$ is 
\begin{equation}
u^{p+1}+v^{p+1}=0.\label{ssp equation}
\end{equation}
Note that this equation guarantees also that the lower $2\times2$
block, representing the action of $F^{2}$ on the $\Sigma$-part of
the Lie algebras is (semi-linearly) nilpotent, i.e. 
\begin{equation}
\left(\begin{array}{ll}
u^{p+1} & uv^{p}\\
vu^{p} & v^{p+1}
\end{array}\right)\left(\begin{array}{ll}
u^{p^{2}(p+1)} & u^{p^{2}}v^{p^{3}}\\
v^{p^{2}}u^{p^{3}} & v^{p^{2}(p+1)}
\end{array}\right)=0.
\end{equation}
We write $\widehat{S}_{ss}=Spf(\kappa[[u,v]]/(u^{p+1}+v^{p+1}))$
for the formal completion of $S_{ss}$ at $x.$ Letting $\zeta$ run
over the $p+1$ roots of $-1$ we recover the $p+1$ formal branches
through $x$ as the ``lines'' 
\begin{equation}
u=\zeta v.
\end{equation}
We write $\widehat{S}_{ss}[\zeta]=Spf(\kappa[[u,v]]/(u-\zeta v))$
for this branch. When we restrict (pull back) the vector bundle $Lie(\mathcal{A})$
to $\widehat{S}_{ss}[\zeta]$, $\ker(F:Lie(\mathcal{A})^{(p)}\rightarrow Lie(\mathcal{A}))$
(the dual of $\omega_{\mathcal{A}}^{(p)}/V(\omega_{\mathcal{A}}),$
which outside $x$ is just $\mathcal{P}_{\mu}^{(p)}$) becomes 
\begin{equation}
\ker\left(\begin{array}{lll}
0 & \zeta v & v\\
\zeta v & 0 & 0\\
v & 0 & 0
\end{array}\right).
\end{equation}
When $v\neq0$ (i.e. outside the point $x$) this is the line (note
again the semi-linearity, and the relation $\zeta^{-p}=-\zeta$) 
\begin{equation}
\kappa\left(\begin{array}{l}
0\\
1\\
\zeta^{-1}
\end{array}\right).
\end{equation}
As these lines are distinct, the filtration of $\mathcal{P}$ can
not be extended across $x.$ \end{proof}

\subsection{The Hasse invariant $h_{\bar{\Sigma}}$}

\subsubsection{Definition of $h_{\bar{\Sigma}}$}

The construction and main properties of the Hasse invariant that we
are about to describe have been given (for any unitary Shimura variety)
by Goldring and Nicole in {[}12{]} (see also {[}22{]}), but in our
case they can be also obtained easily from the discussion of the previous
subsection. Let $R$ be any $\kappa$-algebra, and $(A,\lambda,\iota,\alpha)\in\mathcal{M}(R).$

\begin{definition} The Hasse invariant of $A$
\begin{equation}
h_{\bar{\Sigma}}(A)\in Hom(\omega_{A/R}(\bar{\Sigma}),\omega_{A/R}(\bar{\Sigma})^{(p^{2})})
\end{equation}
is the map $h_{\bar{\Sigma}}=V^{(p)}\circ V.$ \end{definition}

Applying the same definition to the universal semi-abelian scheme
$\mathcal{A}/\bar{S}$ we get (note $\mathcal{L}^{(p^{2})}\simeq\mathcal{L}^{p^{2}}$)
\begin{equation}
h_{\bar{\Sigma}}\in Hom_{\bar{S}}(\mathcal{L},\mathcal{L}^{p^{2}})=H^{0}(\bar{S},\mathcal{L}^{p^{2}-1})=M_{p^{2}-1}(N,\kappa).
\end{equation}
Thus the Hasse invariant is a modular form of weight $p^{2}-1$ defined
over $\kappa.$

\begin{theorem} \label{Hasse Invariant}The Hasse invariant is invertible
on $\bar{S}_{\mu}$ and vanishes on $S_{ss}$ to order one. More precisely,
when we endow $S_{ss}$ with its induced reduced subscheme structure,
it becomes the Cartier divisor $div(h_{\bar{\Sigma}}).$ \end{theorem}

\begin{proof} The Hasse invariant vanishes precisely where $V(\mathcal{L})$
is contained in $\ker(V^{(p)}:\mathcal{P}^{(p)}\rightarrow\mathcal{L}^{(p^{2})}).$
We have already seen that over $\bar{S}_{\mu}$ the latter is the
line bundle $\mathcal{P}_{0}^{(p)}$ and that $V$ sends $\mathcal{L}$
isomorphically onto a direct complement of $\mathcal{P}_{0}^{(p)}$,
\emph{cf }(\ref{P^p}). Thus the Hasse invariant does not vanish on
$\bar{S}_{\mu}.$ To prove that $h_{\bar{\Sigma}}$ vanishes on $S_{ss}$
to order 1 we must study the Dieudonné module at an infinitesimal
neighborhood of a point $x\in S_{gss}$ and compute $V^{(p)}\circ V$
using local coordinates there. In Lemma \ref{Braid} we described
the (contravariant) Dieudonné module at $x.$ In subsection \ref{PD deformation}
below we describe its infinitesimal deformation. Using the local coordinates
$u$ and $v$ introduced there, $f_{3}-uf_{1}-vf_{2}$ becomes a basis
for $\mathcal{L}$ over the first infinitesimal neighborhood of $x.$
We then compute 
\begin{eqnarray}
V(f_{3}-uf_{1}-vf_{2}) & = & e_{1}^{(p)}-ue_{2}^{(p)}\notag\\
V^{(p)}(e_{1}^{(p)}-ue_{2}^{(p)}) & = & -uf_{3}^{(p^{2})}=-u\cdot(f_{3}-uf_{1}-vf_{2})^{(p^{2})}.
\end{eqnarray}
It follows that after $\mathcal{L}$ has been locally trivialized,
the equation $h_{\bar{\Sigma}}=0$ becomes $u=0,$ which is the local
equation for $S_{gss}.$ \end{proof}

\subsubsection{Infinitesimal deformations\label{PD deformation}}

Let $x\in S_{gss}.$ We shall study the infinitesimal deformation
of the Dieudonné module of $\mathcal{A}$ at $x$. Let $\mathcal{O}_{S,x}$
be the local ring of $S$ at $x,$ $\frak{m}$ its maximal ideal,
and $R=\mathcal{O}_{S,x}/\frak{m}^{2}.$ Thus $Spec(R)$ is the first
infinitesimal neighborhood of $x$ in $S,$ and $R\simeq\kappa[u,v]/(u^{2},uv,v^{2}),$
although the choice of the local parameters $u$ and $v$ lies still
at our disposal. The module of Kähler differentials of $R$ is the
3-dimensional vector space over $\kappa$
\begin{equation}
\Omega_{R}^{1}=\Omega_{R/\kappa}^{1}=Rdu+Rdv/\left\langle udu,vdv,udv+vdu\right\rangle .
\end{equation}
Let $A$ be the restriction of $\mathcal{A}$ to $Spec(R)$ and $A_{0}=\mathcal{A}_{x}$
its special fiber. Let $D=H_{dR}^{1}(A/R)$ (a free $R$-module of
rank $6$) and 
\[
D_{0}=H_{dR}^{1}(A_{0}/\kappa)=\kappa\otimes_{R}D,
\]
identified with the Dieudonné module of $A_{0}$ modulo $p$ (see
Lemma \ref{Braid})$.$

The Gauss-Manin connection {[}20{]} is in general defined only for
abelian schemes over a base which is smooth over a field. In our case,
despite the fact that $R$ is not smooth over $\kappa,$ the Gauss-Manin
connection of $\mathcal{A}/S$ yields, by base change, also a connection
\begin{equation}
\nabla:D\rightarrow\Omega_{R}^{1}\otimes_{R}D.
\end{equation}
Caution must be exercised, though, because $\Omega_{R}^{1}\neq\Omega_{S}^{1}\otimes_{\mathcal{O}_{S}}R.$
The Kodaira-Spencer map over $R,$ for example, is not an isomorphism.

We claim that $D$ has a basis of horizontal sections for $\nabla$.
This follows from the crystalline nature of $H^{1},$ but for completeness
we give the easy argument. If $x\in D$ and 
\begin{equation}
\nabla x=du\otimes x_{1}+dv\otimes x_{2}
\end{equation}
($x_{i}\in D$) then $\tilde{x}=x-ux_{1}-vx_{2}$ satisfies 
\begin{equation}
\nabla\tilde{x}=-u\nabla x_{1}-v\nabla x_{2}.
\end{equation}
But if $\nabla x_{1}=du\otimes x_{11}+dv\otimes x_{12}$ and $\nabla x_{2}=du\otimes x_{21}+dv\otimes x_{22}$
then 
\begin{equation}
0=\nabla^{2}x=du\wedge dv\otimes(x_{21}-x_{12})
\end{equation}
hence $x_{21}-x_{12}\in\frak{m}D.$ As $\frak{m}^{2}=0$, $udu=vdv=0$
and $udv=-vdu,$ it follows that 
\begin{eqnarray}
\nabla\tilde{x} & = & -udv\otimes x_{12}-vdu\otimes x_{21}\notag\\
 & = & du\otimes v(x_{12}-x_{21})=0.
\end{eqnarray}
This means that $\tilde{x}$ is a horizontal section having the same
specialization as $x$ in the special fiber, so the horizontal sections
span $D$ over $R$ by Nakayama's lemma.

Let $e_{1},e_{2},f_{3},f_{1},f_{2},e_{3}$ be any six horizontal sections
over $R$, specializing to a basis of $D_{0}$. Identify $D_{0}$
with their $\kappa$-span in $D$. As we have just seen, 
\begin{equation}
R\otimes_{\kappa}D_{0}\rightarrow D
\end{equation}
is surjective, hence, by a dimension count, an isomorphism, and $\nabla=d\otimes1$
on the left hand side. Since $R^{d=0}=\kappa,$ it follows that $D_{0}=D^{\nabla},$
i.e. there are no more horizontal sections besides $D_{0}.$ Thus
every $x_{0}\in$ $H_{dR}^{1}(A_{0}/\kappa)$ has a \emph{unique}
extension to a horizontal section $x\in H_{dR}^{1}(A/R).$

The pairing $\left\langle ,\right\rangle $ on $D$ induced from the
polarization is horizontal for $\nabla,$ i.e. 
\begin{equation}
d\left\langle x,y\right\rangle =\left\langle \nabla x,y\right\rangle +\left\langle x,\nabla y\right\rangle .
\end{equation}

We conclude that we may regard the basis of $D_{0}$ given in Lemma
\ref{Braid} also as a basis of horizontal sections spanning $D$
over $R,$ and that the action of $\iota(\mathcal{O}_{\mathcal{K}})$
and the pairing $\left\langle ,\right\rangle $ are given by the formulae
of the lemma also over $R$.

As $p\neq2,$ $R$ has a canonical divided power structure, and Grothendieck's
crystalline deformation theory {[}16{]} tells us that $A/R$ is completely
determined by $A_{0}$ and by the Hodge filtration $\omega_{A/R}\subset D=R\otimes_{\kappa}D_{0}.$
Since $A$ \emph{is} the universal infinitesimal deformation of $A_{0}$,
we may choose the coordinates $u$ and $v$ so that 
\begin{equation}
\mathcal{P}=Span_{R}\{e_{1}+ue_{3},e_{2}+ve_{3}\}.
\end{equation}
The fact that $\omega_{A/R}$ is isotropic for $\left\langle ,\right\rangle $
implies then that 
\begin{equation}
\mathcal{L}=Span_{R}\{f_{3}-uf_{1}-vf_{2}\}.
\end{equation}

Consider the abelian scheme $A^{(p)}.$ This is \emph{not} the universal
deformation of $A_{0}^{(p)}$ over $R.$ In fact, the $p$-power map
$\phi:R\rightarrow R$ factors as 
\begin{equation}
R\overset{\pi}{\rightarrow}\kappa\overset{\phi}{\rightarrow}\kappa\overset{i}{\rightarrow}R,
\end{equation}
and therefore $A^{(p)},$ \emph{unlike} $A,$ is constant: $A^{(p)}=A_{0}^{(p)}\times_{\kappa}R$
(intuitively, $\Phi$ is contracting on the base, so the pull-back
of $\mathcal{A}$ becomes constant on Artinian neighborhoods which
are contracted to a point). As with $D,$ $D^{(p)}=R\otimes_{\kappa}D_{0}^{(p)},$
$\nabla=d\otimes1,$ but this time the basis of horizontal sections
can be obtained also from the trivalization of $A^{(p)},$ and $\omega_{A^{(p)}/R}=Span_{R}\{e_{1}^{(p)},e_{2}^{(p)},f_{3}^{(p)}\}.$

The isogenies $Frob$ and $Ver,$ like any isogeny, take horizontal
sections with respect to the Gauss-Manin connection to horizontal
sections, e.g. if $x\in D$ and $\nabla x=0$ then $Vx\in D^{(p)}$
satisfies $\nabla(Vx)=0.$ Since $V$ and $F$ preserve horizontality,
$e_{1},f_{2},e_{3}$ span $\ker(V)$ over $R$ in $D,$ and the relations
in \emph{(v)} and \emph{(vi)} of Lemma \ref{Braid} continue to hold.
Indeed, the matrix of $V$ in the basis at $x$ prescribed by that
lemma continues to represent $V$ over $Spec(R)$ by ``horizontal
continuation''. The matrix of $F$ is then derived from the relation
\begin{equation}
\left\langle Fx,y\right\rangle =\left\langle x,Vy\right\rangle ^{(p)}
\end{equation}
($x\in D^{(p)},y\in D$).

The Hodge filtration nevertheless varies, so we conclude that 
\begin{equation}
\mathcal{P}_{0}=\mathcal{P}\cap\ker(V)=Span_{R}\{e_{1}+ue_{3}\}.
\end{equation}
The condition $V(\mathcal{L})=\mathcal{P}_{0}^{(p)},$ which is the
``equation'' of the closed subscheme $S_{gss}\cap Spec(R)$ means
\begin{equation}
V(f_{3}-uf_{1}-vf_{2})=e_{1}^{(p)}-ue_{2}^{(p)}\in R\cdot e_{1}^{(p)}
\end{equation}
and this holds if and only if $u=0.$ We have proved the following
lemma, which completes the proof of Theorem \ref{Hasse Invariant}.

\begin{lemma} Let $x\in S_{gss}$ and the coordinates $u,v$ be as
above. Then the closed subscheme $S_{gss}\cap Spec(R)$ is given by
the equation $u=0.$ \end{lemma}

\subsection{A secondary Hasse invariant on the supersingular locus}

In his forthcoming Ph.D. thesis {[}5{]}, Boxer develops a general
theory of secondary Hasse invariants defined on lower strata of Shimura
varieties of Hodge type. See also {[}21{]}. In this section we provide
an independent approach, in the case of Picard modular surfaces, affording
a detailed study of its properties.

\subsubsection{Definition of $h_{ssp}$}

As we have seen, along the general supersingular locus $S_{gss},$
Verschiebung induces isomorphisms 
\begin{equation}
V_{\mathcal{L}}:\mathcal{L}\simeq\mathcal{P}_{0}^{(p)},\,\,\,V_{\mathcal{P}}:\mathcal{P}_{\mu}\simeq\mathcal{L}^{(p)}.
\end{equation}
(The first is unique to $S_{gss},$ the second holds also on the $\mu$-ordinary
stratum.) Consider the isomorphism 
\begin{equation}
V_{\mathcal{P}}^{(p)}\otimes V_{\mathcal{L}}^{-1}:\mathcal{P}_{\mu}^{(p)}\otimes\mathcal{P}_{0}^{(p)}\simeq\mathcal{L}^{(p^{2})}\otimes\mathcal{L}\simeq\mathcal{L}^{p^{2}+1}.
\end{equation}
Its source is the line bundle $\det\mathcal{P}^{(p)}$ which is identified
with $\mathcal{L}^{(p)}\simeq\mathcal{L}^{p}.$ We therefore get a
nowhere vanishing section 
\begin{equation}
\tilde{h}_{ssp}\in H^{0}(S_{gss},\mathcal{L}^{p^{2}-p+1}).
\end{equation}

Our ``secondary'' Hasse invariant is the nowhere vanishing section
\begin{equation}
h_{ssp}=\tilde{h}_{ssp}^{p+1}\in H^{0}(S_{gss},\mathcal{L}^{p^{3}+1}).
\end{equation}
We shall show that $h_{ssp}$ extends to a holomorphic section on
$S_{ss}$, and vanishes at the superspecial points (to a high order).

\subsubsection{Computations at the superspecial points}

The goal of this subsection is to show that $h_{ssp}$ (unlike $\tilde{h}_{ssp}$)
extends over $S_{ss}$, and to compute its order of vanishing at the
superspecial points. We refer to the computations of Proposition \ref{superspecial}$.$
Dualizing (to put us back in the contravariant world), and using the
letters $e_{i},f_{j}$ to denote \emph{the dual basis} to the basis
used there we get the following.

\begin{lemma} Let $x\in S_{ssp}$ be a superspecial point. There
exist formal coordinates $u$ and $v$ so that the formal completion
of $S$ at $x$ is $\widehat{S}=Spf(\kappa[[u,v]]),$ and $D=H_{dR}^{1}(\mathcal{A}/\widehat{S})$
has a basis $f_{3},e_{1},e_{2},e_{3},f_{1},f_{2}$ over $\kappa[[u,v]]$
with the following properties:

(i) $f_{3},e_{1},e_{2}$ is a basis for $\omega_{\mathcal{A}}$

(ii) The basis is symplectic, i.e. the polarization form is $\left\langle e_{i},f_{j}\right\rangle =-\left\langle f_{j},e_{i}\right\rangle =\delta_{ij},$
$\left\langle e_{i},e_{j}\right\rangle =\left\langle f_{i},f_{j}\right\rangle =0.$

(iii) $\mathcal{O}_{\mathcal{K}}$ acts on the $e_{i}$ via $\Sigma$
and on the $f_{j}$ via $\bar{\Sigma}.$

(iv) $V:D\rightarrow D^{(p)}$ is given by $Vf_{3}=ue_{1}^{(p)}+ve_{2}^{(p)},$
$Ve_{1}=$ $uf_{3}^{(p)},$ $Ve_{2}=vf_{3}^{(p)},$ $Ve_{3}=f_{3}^{(p)},$
$Vf_{1}=e_{1}^{(p)},$ $Vf_{2}=e_{2}^{(p)}.$ \end{lemma}

Using the lemma, we compute along $\widehat{S}_{ss}[\zeta]$, where
$u=\zeta v$ ($\zeta^{p+1}=-1$). See the discussion following (\ref{ssp equation})
for the definition of the formal branch $\widehat{S}_{ss}[\zeta].$
Denote by $\mathcal{P}[\zeta],$ $\mathcal{P}_{0}[\zeta]$ and $\mathcal{L}[\zeta]$
the pull-backs of the corresponding vector bundles to $\widehat{S}_{ss}[\zeta].$
The map $V_{\mathcal{L}}$ is given by 
\begin{equation}
f_{3}\mapsto ue_{1}^{(p)}+ve_{2}^{(p)}=v\cdot(\zeta e_{1}^{(p)}+e_{2}^{(p)})=v\cdot(\zeta^{p}e_{1}+e_{2})^{(p)}\in\mathcal{P}_{0}^{(p)}[\zeta].
\end{equation}
Use $e_{1}\wedge e_{2}=e_{1}\wedge(\zeta^{p}e_{1}+e_{2})$ as a basis
for $\det\mathcal{P}[\zeta].$ Since $V_{\mathcal{P}}^{(p)}$ maps
$e_{1}^{(p)}$ to $(\zeta v)^{p}f_{3}^{(p^{2})},$ $\tilde{h}_{ssp}$
maps $e_{1}^{(p)}\wedge e_{2}^{(p)}=e_{1}^{(p)}\wedge(\zeta^{p}e_{1}+e_{2})^{(p)}$
to 
\begin{eqnarray}
\tilde{h}_{ssp}(e_{1}^{(p)}\wedge e_{2}^{(p)}) & = & \zeta^{p}v^{p-1}f_{3}^{(p^{2})}\otimes f_{3}\notag\\
 & = & \zeta^{p}v^{p-1}f_{3}^{p^{2}+1}=\zeta u^{p-1}f_{3}^{p^{2}+1}.
\end{eqnarray}

\begin{lemma} There does not exist a function $g\in\kappa[[u,v]]/(u^{p+1}+v^{p+1})$
on $\widehat{S}_{ss}$ whose restriction to the branch $\widehat{S}_{ss}[\zeta]$
is $\zeta u^{p-1}.$ \end{lemma}

\begin{proof} Had there been such a function $g$, represented by
a power series $G\in\kappa[[u,v]],$ then we would get $vg=u^{p}$
on $\widehat{S}_{ss}[\zeta]$ for every $\zeta,$ hence 
\begin{equation}
vG-u^{p}\in(u^{p+1}+v^{p+1})\subset\kappa[[u,v]].
\end{equation}
But any power series in the ideal $(u^{p+1}+v^{p+1})$ contains only
terms of degree $\ge p+1,$ while in $vG-u^{p}$ we can not cancel
the term $u^{p}.$ \end{proof}

The lemma means that $\tilde{h}_{ssp}$ can not be extended over $S_{ss}$
to a section of $Hom(\det\mathcal{P}^{(p)},\mathcal{L}^{p^{2}+1})\simeq\mathcal{L}^{p^{2}-p+1}.$
However, when we raise it to a $p+1$ power the dependence on $\zeta$
disappears. It then extends to a section $h_{ssp}$ of $\mathcal{L}^{p^{3}+1}$
over $S_{ss},$ given over $\widehat{S}_{ss}$ (the formal completion
of $S_{ss}$ at $x$) by the equation 
\begin{equation}
h_{ssp}=\varepsilon u^{p^{2}-1}f_{3}^{p^{3}+1},\label{order at ssp}
\end{equation}
where $\varepsilon\in\kappa[[u,v]]^{\times}$ depends on the isomorphism
between $\det\mathcal{P}$ and $\mathcal{L}.$

\begin{theorem} The secondary Hasse invariant $h_{ssp}$ belongs
to $H^{0}(S_{ss},\mathcal{L}^{p^{3}+1}).$ It vanishes precisely at
the points of $S_{ssp}.$ The subscheme $``h_{ssp}=0"$ of $S_{ss}$
is not reduced. At $x\in S_{ssp}$, with $u$ and $v$ as above, it
is the spectrum of 
\begin{equation}
\kappa[[u,v]]/(u^{p+1}+v^{p+1},u^{p^{2}-1},v^{p^{2}-1}).
\end{equation}
\end{theorem}

\section{On the number of supersingular curves on $S$}

We continue to work over $\kappa,$ and to ease the notation drop
the subscript $\kappa$.

\subsection{The connected components of $S_{ss}$}

The Picard surface $\bar{S}$ is not connected. The supersingular
locus $S_{ss}$ is, however, as connected as it could be.

For the next proposition we need, besides the smooth compactification,
also the Baily-Borel (singular) compactification $S^{*}$ of $S,$
see {[}25{]} and {[}3{]}. Every geometric component of $C=\bar{S}-S$
is contracted in $S^{*}$ to a point. The surface $S^{*}$ is known
to be normal.

\begin{proposition} The scheme $S_{\mu}^{*}=S^{*}-S_{ss}$ is affine
and the intersection of $S_{ss}$ with every connected component of
$S$ is connected. \end{proposition}

\begin{proof} The line bundle $\mathcal{L}$ is ample on $S$, even
over $R_{0}$.\footnote{One way to see it is to use the ampleness of the Hodge bundle $\det\omega_{\mathcal{A}}\simeq\mathcal{L}^{2}$
(pull back from Siegel space, where it is known to be ample by {[}Fa-Ch{]}).} Hence for large enough $m,$ which we take to be a multiple of $p^{2}-1,$
$\mathcal{L}^{m}$ is very ample, and by {[}25{]} the Baily-Borel
compactification $S^{*}$ is the closure of $S$ in the projective
embedding supplied by the linear system $H^{0}(S,\mathcal{L}^{m}).$
It follows that $\mathcal{L}^{m}$ has an extension to a line bundle
on $S^{*}$ which we denote $\mathcal{O}_{S^{*}}(1),$ since it comes
from the restriction of the $\mathcal{O}(1)$ of the projective space
to $S^{*}.$ Moreover, Larsen proves that on the smooth compactification
$\bar{S},$ $\mathcal{L}^{m}=\pi^{*}\mathcal{O}_{S^{*}}(1)$ where
$\pi:\bar{S}\rightarrow S^{*}$.\footnote{It is not clear that $\mathcal{L}$ itself has an extension to a line
bundle on $S^{*},$ or that $\pi_{*}\mathcal{L},$ which is a coherent
sheaf extending $\mathcal{L}|_{S},$ is a line bundle (the problem
lying of course only at the cusps). In other words, it is not clear
that we can extract an $m$th root of $\mathcal{O}_{S^{*}}(1)$ as
a line bundle.}

Replacing $h_{\bar{\Sigma}}$ by its power $h_{\bar{\Sigma}}^{m/(p^{2}-1)}$,
this power becomes a global section of $\mathcal{L}^{m}$, hence its
zero locus $S_{ss}$ a hyperplane section of $S^{*}$ in the projective
embeding supplied by $H^{0}(S,\mathcal{L}^{m}).$ Its complement is
therefore affine. The second claim follows from the fact {[}18{]},
III, 7.9, that a positive dimensional hyperplane section of a normal
projective variety is connected. \end{proof}

\subsection{The number of irreducible components}

\subsubsection{The degree of $\mathcal{L}$ along an irreducible component of $S_{ss}$}

Assume that $N$ is large enough (depending on $p$) so that Theorem
\ref{Vollaard}(iii) holds. Each irreducible component $Z$ of $S_{ss}$
is non-singular, and the secondary Hasse invariant $h_{ssp}$ has
a zero of order $p^{2}-1$ at every superspecial point of $Z$, as
follows from (\ref{order at ssp}). Each component contains $p^{3}+1$
superspecial points. It follows that if $Z$ is such a component,
\begin{equation}
\deg(\mathcal{L}^{p^{3}+1}|_{Z})=\deg(div_{Z}(h_{ssp}))=(p^{3}+1)(p^{2}-1).
\end{equation}
We have proved the following lemma.

\begin{lemma} Let $Z$ be an irreducible component of $S_{ss},$
and assume that $N$ is large enough. Then $\deg(\mathcal{L}|_{Z})=p^{2}-1.$
\end{lemma}

\subsubsection{A computation of intersection numbers}

Let 
\begin{equation}
Z=\bigcup_{i=1}^{n}Z_{i}
\end{equation}
be the decomposition into irreducible components of a single connected
component $Z$ of $S_{ss}$ (recall that $Z$ is the intersection
of $S_{ss}$ with a connected component of $\bar{S}$). If $N$ is
large, then the $Z_{i}$ are smooth, and as they are Fermat curves
of degree $p+1$, their genus is $g(Z_{i})=p(p-1)/2.$

\begin{theorem} \label{ss curves}Let $c_{2}$ be the Euler characteristic
of the connected component of $\bar{S}$ containing $Z,$ i.e. if
over $\Bbb{C}$ this connected component is $\bar{X}_{\Gamma}$ then
\begin{equation}
c_{2}=\sum_{i=0}^{4}(-1)^{i}\dim_{\Bbb{C}}H^{i}(\bar{X}_{\Gamma},\Bbb{C}).
\end{equation}
Then the number $n$ of irreducible components of $Z$ is given by
\begin{equation}
3n=c_{2}.
\end{equation}
\end{theorem}

This $c_{2}=c_{2}(\bar{X}_{\Gamma})$ is given by Holzapfel's formula.
To get the theorem quoted in the introduction, one would have to sum
over all the connected components.

\begin{proof} We first prove the theorem under the assumption that
$N$ is large enough. Computing intersection numbers, and using the
fact that every $Z_{i}$ meets transversally $(p^{3}+1)p$ other $Z_{j}^{\prime}s$
(there are $p^{3}+1$ intersection points on $Z_{i},$ through each
of which pass $p+1$ components, including $Z_{i}$ itself), we get
\begin{equation}
(Z.Z)=n(p^{3}+1)p+\sum_{i=1}^{n}(Z_{i}.Z_{i}).
\end{equation}
Denote by $K_{\bar{S}}$ a canonical divisor on the given connected
component of $\bar{S}.$ From the adjunction formula, 
\begin{equation}
p(p-1)-2=2g(Z_{i})-2=Z_{i}.(Z_{i}+K_{\bar{S}}).
\end{equation}
As we have seen in Proposition \ref{canonical class}, $\mathcal{O}(K_{\bar{S}}+C)\simeq\mathcal{L}^{3}$
where $C$ is the cuspidal divisor (on the given connected component
of $\bar{S})$. Hence 
\begin{equation}
(Z_{i}.K_{\bar{S}})=Z_{i}.(K_{\bar{S}}+C)=\deg(\mathcal{L}^{3}|_{Z_{i}})=3(p^{2}-1)
\end{equation}
by the previous lemma. We get 
\begin{equation}
(Z_{i}.Z_{i})=-2p^{2}-p+1.
\end{equation}
Plugging this into the expression for $(Z.Z)$ we get 
\begin{equation}
(Z.Z)=n(p^{2}-1)^{2}.
\end{equation}

On the other hand, as $Z$ is the divisor of the Hasse invariant on
the given connected component of $\bar{S}$, $div(h_{\bar{\Sigma}})=Z,$
and $h_{\bar{\Sigma}}$ is a global section of $\mathcal{L}^{p^{2}-1},$
we get $\mathcal{O}(Z)=\mathcal{L}^{p^{2}-1}$. From the relation
$(Z.Z)=c_{1}(\mathcal{O}(Z))^{2}$ between the self-intersection number
and the first Chern class, 
\begin{equation}
n=c_{1}(\mathcal{L})^{2}.
\end{equation}
From this and $\mathcal{O}(K_{\bar{S}}+C)\simeq\mathcal{L}^{3}$ we
get $9n=(K_{\bar{S}}+C).(K_{\bar{S}}+C).$ Holzapfel {[}19{]} (4.3.11$^{\prime}$)
on p.184, implies 
\begin{equation}
9n=3c_{2}(X_{\Gamma})=3c_{2}(\bar{X}_{\Gamma})
\end{equation}
proving the theorem when $N$ is large.

We now note that while the assumption of $N$ being large was crucial
for the intersection-theoretic computations, the end result $3n=c_{2}$
holds for a given $N\ge3$ if and only if it holds for any multiple
$N^{\prime}$ of $N.$ Indeed, the covering $S(N^{\prime})\rightarrow S(N)$
is étale, say of degree $d(N,N^{\prime}).$ The second Chern class,
being equal to the Euler characteristic, gets multiplied by $d(N,N^{\prime}).$
But thanks to Theorem \ref{Vollaard}(iv), the same holds true for
the number $n.$ We may therefore deduce the validity of our formula
for $N$ from its validity for $N^{\prime}.$ This completes the proof
of the theorem. \end{proof}

One can probably get a ``mass formula'' for the number of irreducible
components weighted by the reciprocals of the orders of certain automorphism
groups even if $N=1.$ We do not pursue it here.

We easily deduce the following two corollaries.

\begin{corollary} (a) Let $p$ be inert in $\mathcal{K}.$ Then for
$N$ sufficiently large the number of superspecial points on $S$
is 
\begin{equation}
\frac{c_{2}(\overline{S})}{3}\cdot(p^{2}-p+1).
\end{equation}

(b) The aritmetic genus $g_{a}$ of a connected component $Z$ of
$S_{ss}$ is given by 
\begin{equation}
g_{a}(Z)=\frac{c_{2}(\overline{S}_{Z})}{3}\cdot\frac{(p^{4}+p^{2}-2)}{2}+1,
\end{equation}
where $\overline{S}_{Z}$ is the connected component of $\overline{S}$
containing $Z.$ \end{corollary}

\begin{proof} (a) Each irreducible component of $S_{ss}$ contains
$p^{3}+1$ points of $S_{ssp}$ and each point of $S_{ssp}$ is shared
by $p+1$ irreducible components.

(b) This follows from the adjunction formula, which for the singular
curve $Z$ takes the form 
\begin{equation}
2g_{a}(Z)-2=Z.(Z+K_{\bar{S}})
\end{equation}
(see {[}27{]}, Chapter 9, Theorem 1.37), using the computations of
$Z.Z$ and $Z.K_{\bar{S}}$ carried out in the proof of the theorem.
\end{proof}

\bigskip{}

\textbf{Bibliography\medskip{}
}

{[}1{]} F. Andreatta, E. Z. Goren: Hilbert modular forms: mod $p$
and $p$-adic aspects, Memoirs A.M.S. \textbf{819, }2005.\medskip{}

{[}2{]} E. Bachmat, E. Z. Goren: On the non ordinary locus in Hilbert-Blumenthal
surfaces, Math. Annalen, \textbf{313} (1999) 475-506.\medskip{}

{[}3{]} J. Bella\"{i}che: Congruences endoscopiques et représentations
Galoisiennes, \emph{Thèse}, Paris XI (Orsay), 2002.\medskip{}

{[}4{]} J. Bella\"{i}che: Relévement des formes modulaires de Picard,
J. London Math. Soc.(2), \textbf{74 }(2006), 13-25.\medskip{}

{[}5{]} G. Boxer, \emph{Ph.D. Thesis, }Harvard, to appear.\medskip{}

{[}6{]} O.Bültel, T.Wedhorn: Congruence relations for Shimura varieties
associated to some unitary groups, J. Instit. Math. Jussieu \textbf{5},
(2006) 229-261.\medskip{}

{[}7{]} J. Cogdell: Arithmetic cycles on Picard modular surfaces and
modular forms of nebentypus, J. Reine Angew. Math. \textbf{357} (1985),
115-137.\medskip{}

{[}8{]} P. Deligne: Travaux de Shimura, Sém. Bourbaki \textbf{389}
(1971).\medskip{}

{[}9{]} P. Deligne, M. Rapoport: Les schémas de modules de courbes
elliptiques, \emph{in:} \emph{Modular Functions of One Variable} II,
LNM \textbf{349} (1973), 143-316.\medskip{}

{[}10{]} M. Demazure: \emph{Lectures on} $p$\emph{-divisible groups},
LNM \textbf{302}, Springer, 1972.\medskip{}

{[}11{]} G. Faltings, C.-L. Chai: \emph{Degeneration of abelian varieties,
}Springer, 1990.\medskip{}

{[}12{]} W. Goldring, M.-H. Nicole: A $\mu$-ordinary Hasse invariant,
\emph{preprint}, arXiv:1302.1614.\medskip{}

{[}13{]} B. Gordon: Canonical models of Picard modular surfaces, \emph{in:
}{[}24{]}, 1-29.\medskip{}

{[}14{]} E. Z. Goren: Hasse invariants for Hilbert modular varieties,
Israel J. Math. \textbf{122} (2001), 157-174.\medskip{}

{[}15{]} E. Z. Goren, F. Oort: Stratification of Hilbert modular varieties,
J. Alg. Geom. \textbf{9} (2000), 111-154.\medskip{}

{[}16{]} A. Grothendieck: \emph{Groupes de Barsotti-Tate et cristaux
de Dieudonné}, Les Presses de l'Université de Montréal, 1974.\medskip{}

{[}17{]} M. Harris: Arithmetic vector bundles and automorphic forms
on Shimura varieties. I, Invent. Math. \textbf{82} (1985), 151-189.\medskip{}

{[}18{]} R. Hartshorne: \emph{Algebraic Geometry,} Graduate Texts
in Mathematics, No. \textbf{52}, Springer, 1977.\medskip{}

{[}19{]} R.-P. Holzapfel: \emph{Ball and surface arithmetics}, Aspects
of Mathematics, \textbf{E29}, Vieweg \& Sohn, 1998.\medskip{}

{[}20{]} N. Katz, T. Oda: On the differentiation of de Rham cohomology
classes with respect to parameters, J. Math. Kyoto Univ. \textbf{8}
(1968), 199-213.\medskip{}

{[}21{]} J.-S. Koskivirta: Sections of the Hodge bundle over Ekedahl
Oort strata of Shimura varieties of Hodge type, \emph{preprint}, arXiv:1410.1317.\medskip{}

{[}22{]} J.-S. Koskivirta, T. Wedhorn: Hasse invariants for Shimura
varieties of Hodge type, \emph{preprint,} arXiv:1406.2178.\medskip{}

{[}23{]} K.-W. Lan: Arithmetic compactifications of PEL-type Shimura
varieties, London Mathematical Society Monographs \textbf{36}, Princeton,
2013.\medskip{}

{[}24{]} R. Langlands, D. Ramakrishnan: \emph{The zeta functions of
Picard modular surfaces}, Univ. Montréal, Montréal, 1992.\medskip{}

{[}25{]} M. Larsen: Unitary groups and $l$-adic representations,
\emph{Ph.D. Thesis}, Princeton University, 1988.\medskip{}

{[}26{]} M. Larsen: Arithmetic compactification of some Shimura surfaces
\emph{in}: {[}24{]}, 31-45.\medskip{}

{[}27{]} Q. Liu: \emph{Algebraic geometry and arithmetic curves},
Oxford University Press, 2002.\medskip{}

{[}28{]} J. Milne: \emph{Introduction to Shimura varieties,} available
on the web at:\linebreak{}
www.jmilne.org/math/xnotes/svi.pdf.\medskip{}

{[}29{]} T. Oda: On the first de Rham cohomology group and Dieudonné
modules, Ann. Sci. ENS \textbf{2 }(1969), 63-135.\medskip{}

{[}30{]} G. Shimura: The arithmetic of automorphic forms with respect
to a unitary group, Ann. of Math. \textbf{107} (1978), 569-605.\medskip{}

{[}31{]} I. Vollaard: The supersingular locus of the Shimura variety
for $GU(1,s)$, Canad. J. Math. \textbf{62} (2010), 668-720.\medskip{}

{[}32{]} T. Wedhorn: The dimension of Oort strata of Shimura varieties
of PEL-type, \emph{in: Moduli of abelian varieties,} Progr. Math.
\textbf{195}, 441-471, Birkhäuser, 2001.\medskip{}

\end{document}